\documentclass[11pt,a4paper]{amsart}
\usepackage{amssymb, amstext, amscd, amsmath}
\usepackage{url}
\usepackage{amsfonts}
\usepackage{lscape}
\usepackage{amsthm}
\usepackage{amsfonts}
\usepackage{array}
\usepackage{latexsym}
\usepackage{mathrsfs}
\usepackage{textcomp}
\usepackage{verbatim}
\usepackage{tikz}
\usepackage{ifthen}
\usepackage{thmtools, thm-restate}

\newtheorem{thm}{Theorem}[section]

\newtheorem{claim}[thm]{Claim}

\newtheorem{cor}[thm]{Corollary}

\newtheorem{lem}[thm]{Lemma}

\newtheorem{rem}[thm]{Remark}
\numberwithin{equation}{section}

\newcommand{\R}{{\mathbb{R}}}

\newcommand{\bZ}{{\mathbb{Z}}}

  \newcommand{\C}{{\mathcal{C}}}
  
  \newcommand{\E}{{\mathcal{E}}}
  
  \newcommand{\G}{{\mathcal{G}}}

  \newcommand{\I}{{\mathcal{I}}}

  \newcommand{\M}{{\mathcal{M}}}

  \newcommand{\calR}{{\mathcal{R}}}

  \newcommand{\X}{{\mathcal{X}}}

\newcommand{\val}{\mathrm{val}}

\usetikzlibrary{patterns,calc}
\tikzset{roundnode/.style={circle,draw=black!50,fill=black!20,inner sep=1.2pt}}
\tikzset{every loop/.style={}}
\tikzset{polarshift/.style args={#1/#2}{xshift=#1*cos(#2),yshift=#1*sin(#2)}}

\allowdisplaybreaks
\begin{document}
%%%%%%%%%%%%%%%%%%%%%%%%%%%%%%%%%%%%%%%%%%%%%%%%%%%%%%%%%%%%%%%%%%%%%%%%%%%%%%%%%%%%%%%%%%%%%%%%%%%%%%%%%%%%%%%%%%%%%%%%%%%%%%%%%%%%%%%%%%%%%%%%%%%%%%

\title{On Generic Linearly Constrained Frameworks}

\author{Zakir Deniz, Hakan Guler and Anthony Nixon}

\address{Department of Mathematics, D\"uzce University, D\"uzce, 81620, T{\"{u}}rk{\.{\.i}}ye.}
\email{zakirdeniz@duzce.edu.tr}

\address{Department of Mathematics, Kastamonu University, Kastamonu, 37150, T{\"{u}}rk{\.{\.i}}ye.}
\email{hguler@kastamonu.edu.tr}

\address{School of Mathematical Sciences, Lancaster University, Lancaster, LA1 4YF, U.K.}
\email{a.nixon@lancaster.ac.uk}

\date{\today}

\begin{abstract}
A linearly constrained framework in $\mathbb{R}^d$ is a bar-joint framework where, in addition, vertices with loops are constrained to lie in given affine subspaces. In the generic case, when each vertex is incident to sufficiently many loops, a characterisation of rigidity was obtained in \cite{JNT} for all $d\geq 3$. By extending this to characterise the rank function of the linearly constrained rigidity matroid (under the same loop hypothesis), sufficient conditions for a looped simple graph to be (globally) rigid in $\R^d$ are obtained. 
In the 2-dimensional case generic rigidity was characterised in \cite{ST}, and we obtain a sharper sufficient condition in this case. A key technique is the application of the discharging method.
\end{abstract}


\maketitle

\section{Introduction}

A bar-joint framework $(G,p)$ in $\mathbb{R}^d$ is the combination of a finite, simple graph $G$ and a realisation $p:V\rightarrow \mathbb{R}^d$ which assigns positions to the vertices (joints) and hence lengths to the edges (bars). The framework is rigid if the only bar-length-preserving continuous deformations of the joints arise from isometries of $\mathbb{R}^d$, and flexible otherwise. Furthermore, the framework $(G,p)$ is globally rigid if every other framework $(G,q)$ in $\mathbb{R}^d$ with the same edge lengths arises from $(G,p)$ by a composition of isometries. 

It is well known that testing the rigidity or global rigidity of a given framework is computationally challenging \cite{A,Sax} and depends both on the graph $G$ and the specific realisation $p$. Hence, as is common in the literature, we restrict our attention to generic frameworks where rigidity depends only on the underlying graph.
A graph is \emph{$d$-rigid} (resp.\ \emph{globally $d$-rigid}) if there exists a generic (defined formally below) realisation in $\mathbb{R}^d$ that is rigid (resp. globally rigid).

In many practical applications determining the (global) rigidity or flexibility of a structure or its constituent parts is crucial to understanding the form and function of the structure. However often, such as in mechanical engineering, the structure is grounded or the object is restricted to move along a wall or a wire. This motivates the study of linearly constrained frameworks. These are bar-joint frameworks where the graph allows loops and the loops are realised as affine subspaces of $\mathbb{R}^d$ to which the joints are then constrained to move only within these fixed subspaces. In this way the isometries of $\mathbb{R}^d$ can be ruled out and rigidity corresponds to a locally unique solution to the combination of bar-length and affine subspace constraints. Moreover global rigidity corresponds to a globally unique solution. Formal definitions will follow subsequently.

Previous papers on linearly constrained frameworks have given a precise combinatorial description of 2-rigidity for generic frameworks \cite{ST}, extended this to all dimensions under certain hypotheses on the affine subspace constraints \cite{CGJN,JNT}, extended the first result to global 2-rigidity \cite{GJN}, provided a sufficient connectivity condition for both 2-rigidity and global 2-rigidity \cite{G}, and characterising global $d$-rigidity when every vertex is constrained to a fixed line \cite{CMMNT}. Nevertheless giving a general graph theoretic characterisation of generic (global) $d$-rigidity remains an open problem, with the special case when there are no affine subspace constraints corresponding to the well studied (global) $d$-rigidity problem. In this article we give sufficient graph theoretic conditions that guarantee generic (global) $d$-rigidity for linearly constrained frameworks in all dimensions under additional hypotheses on the loops, as well giving a more refined sufficient condition in 2-dimensions. 

Throughout the paper we will use the term graph to describe a graph which may contain multiple edges and loops and denote such a graph by $G = (V, E, L)$ where $V, E, L$ are the sets of vertices, edges and loops, respectively. We will use the term looped simple graph to describe a graph which contains no multiple edges but may contain loops (even multiple loops), and the term simple graph for a graph which contains no loops or multiple edges. We will also use the notation $G^{[t]}$ for a looped simple graph which is obtained from a looped simple graph $G=(V,E,L)$ by adding $t\geq 0$ loops to each vertex.

Let $G=(V,E,L)$ be a looped simple graph. 
We say $G$ is {\em $k$-balanced} if every connected
component of $G-T$, where $T\subseteq V$ with $|T|\leq k$, has at least $k-|T|$ vertices each incident to at least one loop. The following theorem, which can be regarded as the linearly constrained version of Lov\'asz and Yemini's 6-connectivity result for 2-dimensional bar-and-joint frameworks \cite{LY}, says that, for a looped simple graph, being balanced `enough' guarantees 2-rigidity.

\begin{thm}[\cite{G}]\label{thm:2D-6balanced}
Let $G=(V,E,L)$ be a 6-balanced looped simple graph and $F\subseteq E\cup L$ with $|F|\leq 3$. Then $G-F$ is 2-rigid as a linearly constrained framework.
\end{thm}

A graph $G$ is called {\em weakly $k$-balanced} if every connected
component of $G-T$, where $T\subseteq V$ with $|T|\leq k$, has at least $k-|T|$ loops. Clearly, $k$-balancedness implies weakly $k$-balancedness. We will use this weaker definition instead and give an analogous sufficient condition for all $d\geq 2$, that is Theorem \ref{thm:d_geq2_main}, by first characterising the rank function of this matroid (when each vertex has enough loops) in Theorem \ref{thm:rank_cover-both}. That will allow us to deduce our first main result.

\begin{thm}\label{thm:d_geq2_main}
    Let $d,t$ be positive integers with $d\geq 2$, $d\geq 2t-1$ and $G=(V,E,L)$ be a weakly $2t$-balanced looped simple graph. Let $F\subset E\cup L$ with $|F|\leq t$. Then $G^{[d-t]}-F$ is $d$-rigid as a linearly constrained framework. 
\end{thm}

We remark that the case when $d=1$ is simpler and it is straightforward to verify that a looped simple graph in which every connected component contains a loop is 1-rigid. In other words every weakly 1-balanced looped simple graph is 1-rigid.

In the 2-dimensional case we shall go one step further and show in Theorem \ref{thm:weakly6balanced} that the same conclusion holds without the additional hypothesis on loops. That is, for an arbitrary looped simple graph weakly 6-balancedness is sufficient for 2-rigidity. This strenghens Theorem \ref{thm:2D-6balanced} by showing that `$6$-balanced' can be replaced by weakly 6-balanced.

\begin{thm}\label{thm:weakly6balanced}
Let $G=(V,E,L)$ be a weakly 6-balanced looped simple graph and $F\subseteq E\cup L$ with $|F|\leq 3$. Then $G-F$ is 2-rigid as a linearly constrained framework.
\end{thm}

We then strengthen this result further by taking inspiration from the bar-joint case.
Jackson and Jord\'an \cite{JJ} showed that 
every 6-connected graph is globally 2-rigid as a bar-joint framework. 
In \cite{GMRWY}, the authors relaxed this connectivity for global 2-rigidity. A graph $G$ with at least $k+1$ vertices is {\em essentially $k$-connected} if there is no $X\subseteq V(G)$ with $|X|<k$
such that at least two components of $G-X$ are nontrivial, where a nontrivial component means it contains at least one edge. They showed the following.

\begin{thm}[\cite{GMRWY}]\label{thm:4connected6essential}
Every 4-connected and essentially 6-connected graph is globally 2-rigid as a bar-and-joint framework.
\end{thm}

As given in \cite{G}, Theorem \ref{thm:2D-6balanced} and the characterisation of linearly constrained global 2-rigidity given below in Theorem \ref{thm:glob_char} immediately imply that every 6-balanced looped simple graph is globally 2-rigid as a linearly constrained framework.
We shall relax this balancedness condition by first using the weaker version and then by adding an extra condition. For this extra condition we will transfer the essential connectivity concept to looped simple graphs via balancedness and say a looped simple graph $G=(V,E,L)$ is called {\em essentially $k$-balanced} if $|V|\geq k+1$ and for every $T\subset V$ with $|T|<k$, each component of $G-T$ with at least one simple edge has at least one loop. The key difference between (weak) balancedness and essential balancedness is that essential balancedness ignores the connected components of $G-T$ with only one vertex. With this definition we can state our third main result, the linearly constrained analogue of Theorem \ref{thm:4connected6essential}.

\begin{thm}\label{thm:main}
Let $G=(V,E,L)$ be a weakly 4-balanced and essentially 6-balanced looped simple graph.
Then $G$ is globally 2-rigid as a linearly constrained framework. 
\end{thm}

We use the discharging method as the main tool for our proofs.
The discharging method is famous for its role in the proof of the four color theorem for planar graphs. We were inspired by \cite{GMRWY}, which marked the first application of this method in rigidity, but the method seems to be very suitable for the bounded graphs that arise often in rigidity theory.

\section{Preliminaries}

We begin with the relevant background from rigidity theory.

\subsection{Rigid graphs}
A framework $(G,p)$ in $\mathbb{R}^d$ is \emph{generic} if the coordinates of $p$ (as a set) are algebraically independent over the rationals. By focussing on generic frameworks we may make use of a linearisation known as infinitesimal rigidity \cite{AR} and hence study matrices and matroids.

Define the \emph{rigidity matrix} $R_d(G,p)$ of the framework $(G,p)$ to be the $|E|\times d|V|$ matrix whose rows are indexed by the edges and columns indexed by $d$-tuples of the vertices. The row for an edge $e=uv$ is given by:
$$\begin{pmatrix} 0 & \dots & 0 & p(u)-p(v) & 0 & \dots & 0 & p(v)-p(u) & 0 & \dots  \end{pmatrix}$$ 
where $p(u)-p(v)$ occurs in the $d$-tuple of columns indexed by $u$, $p(v)-p(u)$ occurs in the $d$-tuple of columns indexed by $v$ and $p(u),p(v)\in \mathbb{R}^{d}$.
It is easy to see that $\mbox{rank } R_d(G,p)\leq d|V|-\binom{d+1}{2}$ whenever $p$ affinely spans $\mathbb{R}^d$. 

The {\em $d$-dimensional rigidity matroid} of a graph $G=(V,E)$ is the matroid $\mathcal{R}_d(G)$ on $E$ in which a set of edges $F\subseteq E$ is independent whenever the corresponding rows of $R_d(G,p)$ are independent, for some (or equivalently every) generic $p$. 
We denote the rank function of $\mathcal{R}_d(G)$ by $r_d$ and put $r_d(G)=r_d(E)$.

It is easy to see that, when $|V(G)|\geq d+1$, an independent set $F$ in $\mathcal{R}_d(G)$ contains at most $d|V(F)|-\binom{d+1}{2}$ edges.
When $d=1$, it is an elementary folklore result that $\mathcal{R}_1(G)$ is precisely the cycle matroid of $G$. 
When $d=2$, \cite{PG} (as rediscovered in \cite{L}) characterised 2-rigid graphs and \cite{LY} extended this to characterise the rank function of $\mathcal{R}_2(G)$. However for $d\geq 3$ no co-NP characterisation of $d$-rigidity is known.

\subsection{Linearly constrained frameworks}
A {\em linearly constrained framework in $\R^d$}  is a triple $(G,
p, q)$ where $G=(V,E,L)$ is a looped simple graph, $p:V\to \R^d$ and
$q:L\to \R^d$ are maps. For $v_i\in V$ and $e_j\in L$ we put $p(v_i)=p_i$ and
$q(e_j)=q_j$.

An {\em infinitesimal motion} of $(G, p, q)$ is a map $\dot
p:V\to \R^d$ satisfying the system of linear equations:
\begin{eqnarray}
\label{eqn1} (p_i-p_j)\cdot (\dot p_i-\dot p_j)&=&0 \mbox{ for all $v_iv_j \in E$}\\
\label{eqn2} q_j\cdot \dot p_i&=&0 \mbox{ for all incident pairs $v_i\in V$ and
$e_j \in L$.}
\end{eqnarray}
The second constraint implies that  the infinitesimal velocity of each
$v_i\in V$ is constrained to lie on the hyperplane $\mathcal{H}$ through $p_i$ with normal $q_j$
for each loop $e_j$ incident to $v_i$. Let $\mathcal{L}_d(G,p,q)$ be the set of all such hyperplanes $\mathcal{H}$ in $\mathbb{R}^d$ for the linearly constrained framework $(G,p,q)$. Throughout the paper when using this notation the graph will be clear, and a generic framework considered, hence we will simply use $\mathcal{L}_d$.

The {\em rigidity matrix $R_d (G, p, q)$} of the linearly constrained framework $(G, p, q)$ is  the
matrix of coefficients of this system of equations for the unknowns
$\dot p$. Thus $R_d (G, p, q)$ is a $(|E|+|L|)\times d|V|$ matrix, in
which: the row indexed  by an edge $v_iv_j\in E$ has $p(u)-p(v)$ and
$p(v)-p(u)$ in the $d$ columns indexed by $v_i$ and $v_j$,
respectively and zeros elsewhere; the row indexed  by a loop
$e_j=v_iv_i\in L$ has $q_j$  in the $d$ columns indexed by $v_i$ and
zeros elsewhere. The $|E|\times d|V|$ sub-matrix consisting of the rows indexed by $E$ is the {\em bar-joint rigidity matrix} $R_d(G-L,p)$ of the bar-joint framework $(G-L,p)$. The linearly constrained framework $(G,p,q)$ is called \emph {infinitesimally rigid} if the rank of $R_d(G,p,q)$ is $d|V|$ (equivalently if the only infinitesimal motion of $(G,p,q)$ is the zero motion).

Two $d$-dimensional linearly constrained frameworks $(G,p,q)$ and
$(G,\tilde p,q)$ are {\em equivalent} if
\begin{eqnarray*}
  \|p_i-p_j\|^2&=&\|\tilde p_i-\tilde p_j\|^2 \mbox{ for all $v_iv_j \in E$, and}\\
p_i\cdot q_j&=&\tilde p_i\cdot q_j \mbox{ for all incident pairs $v_i\in V$ and $e_j \in L$.}  
\end{eqnarray*}
We say that $(G,p,q)$ is \emph{globally $\mathcal{L}_d$-rigid} if its only equivalent
framework is itself.

The {\em $d$-dimensional linearly constrained rigidity matroid} of a graph $G=(V,E)$ is the matroid $\mathcal{R}_d^{lc}(G)$ on $E\cup L$ in which a set of edges $F\subseteq E\cup L$ is independent whenever the corresponding rows of $R_d(G,p,q)$ are independent, for some (or equivalently every) generic $(p,q)$. 
We denote the rank function of $\mathcal{R}_d^{lc}(G)$ by $r_d^{lc}$ and put $r_d^{lc}(G)=r_d^{lc}(E\cup L)$ and say $G$ is {\em $\mathcal{L}_d$-rigid} if $r_d^{lc}(G)=d|V|$. 
It is easy to see that an independent set $F$ in $\mathcal{R}_d^{lc}(G)$ contains at most $d|V(F)|$ edges.
We also say that $G$ is {\em redundantly $\mathcal{L}_d$-rigid} if $G-f$ is $\mathcal{L}_d$-rigid for all $f\in E\cup L$.

It will be important for us to understand the rank function of $\mathcal{R}_d^{lc}(G)$ in certain cases. To achieve this we will make use of the following characterisation of $\mathcal{L}_d$-rigidity that extended the one in \cite{CGJN}.
Here, for a positive integer $t$, a graph $G=(V,E)$ is \emph{$t$-sparse} if the subgraph $G[X]$ induced by any $X\subset V$ has at most $t|X|$ edges, and \emph{$t$-tight} if it is $t$-sparse and $|E|=t|V|$.

\begin{thm}[\cite{JNT}]\label{thm:d-dimchar}
Let $G=(V,E,L)$ be a looped simple graph such that each vertex is incident with at least $\lfloor \frac{d}{2}\rfloor$ loops. Then $G$ is $\mathcal{L}_d$-rigid if and only if $G$ contains a spanning $d$-tight, $K_{d+2}$-free subgraph with the property that every vertex of $H$ is incident with at least $\lfloor \frac{d}{2}\rfloor$ loops.   
\end{thm}

For linearly constrained frameworks we have the following characterisation of global $\mathcal{L}_2$-rigidity.

\begin{thm}[\cite{GJN}]\label{thm:glob_char}
A looped simple graph $G$ is globally $\mathcal{L}_2$-rigid if and only if
\begin{itemize}
\item[](i) each connected component of $G$ is either a single vertex with two loops or is redundantly $\mathcal{L}_2$-rigid, and
\item[](ii) for all $X\subseteq V(G)$ with $|X|=2$, each connected component of $G-X$ has at least one loop.\footnote{The condition (ii) in Theorem \ref{thm:glob_char} is called ``2-balancedness'' (or ``balancedness'' for short) in \cite{GJN}. Since we reserve the term ``balancedness'' for a different concept, we wrote this condition explicitly in order not to cause any confusion.}
\end{itemize}
\end{thm}

We will also need the following operation on graphs.
Let $H=(V,E,L)$ be a looped simple graph, $d\geq 1$ and $0\leq k\leq d$ be integers. The {\em $d$-dimensional $k$-loop extension} operation forms a new graph $G$ from $H$ by deleting $k$ loops incident to distinct vertices of $H$ and adding a new vertex $v$ and $d + k$ new edges and loops incident to $v$, provided that at least k loops are added at $v$ and exactly one new edge is added from $v$ to each of the end-vertices of the $k$ deleted loops.

\begin{lem}[\cite{CGJN}]\label{lem:ddim-kloop}
Suppose that $G$ is obtained from $H$ by a $d$-dimensional $k$-loop extension operation which deletes a loop $l_j$ at $k$ distinct vertices $v_j$, $1 \leq  j\leq  k$, of $H$ and adds a
new vertex $v$. Suppose further that $H$ is $\mathcal{L}_d$-rigid
and that $v_j$ is incident with at least $\lceil \frac{(k-1)d}{k}\rceil$ loops in $H$ for all $1 \leq j \leq k$ when $k\leq  2$. Then $G$ is $\mathcal{L}_d$-rigid.
\end{lem}

\section{Characterising The Rank Function}\label{sec:rank_func}

In this section we prove the key technical tool we require for the proof of Theorem \ref{thm:d_geq2_main}. This is Theorem \ref{thm:rank_cover-both} below which characterises the rank function of $\mathcal{R}_d^{lc}(G)$ under the additional hypotheses on the loops at each vertex required for us to be able to apply Theorem \ref{thm:d-dimchar}. First we recall the concept of a polymatroid.

\subsection{Polymatroids}
A {\em polymatroid} is a pair $(E,r)$ where $E$ is a finite set and $r:2^E\rightarrow\bZ$ is a function satisfying the following properties
\begin{itemize}
    \item $r$ is {\em submodular}, that is $r(X)+r(Y)\geq r(X\cup Y)+r(X\cap Y)$ for all $X,Y\subseteq E$,
    \item $r$ is {\em monotone}, that is $r(X)\leq r(Y)$ for all $X\subseteq Y$,
    \item $r$ is {\em normalized}, that is $r(\emptyset)=0$.
\end{itemize}

The next classical result of Edmonds and Rota explains the relation between a polymatroid and a matroid.

\begin{thm}[\cite{EdRo}]\label{thm:polymatroid}
    For a polymatroid $(E,f)$, the collection
    $$
    \I_f:=\{F\subseteq E: |I|\leq f(I)\text{ for all } I\subseteq F\}
    $$
    is the collection of independent sets of a matroid on $E$ with the rank function
    $$
    r_f(F)=\min\{|F\setminus I|+f(I):I\subseteq F\}.
    $$
\end{thm}
The matroid $(E,\I_f)$ in Theorem \ref{thm:polymatroid} is called the {\em matroid induced by} $f$.

\subsection{The Rank Function}
Given a looped simple graph $G=(V,E,L)$, let us define two functions $f_i:2^{E\cup L}\rightarrow \bZ$, $i\in \{0,1\}$ by
$$
f_0(T):=t|V(T)|
$$
and
$$
f_1(T):=\begin{cases}
|T|, & T\subseteq E \text{ and } |V(T)|\leq 2t\\
t|V(T)|-1, & T\subseteq E \text{ and } |V(T)|= 2t+1\\
t|V(T)|,& (T\subseteq E \text{ and }|V(T)|\geq 2t+2) \text{ or }T\cap L\neq \emptyset
\end{cases}
$$
where $t>i$ is an integer. 

\begin{rem}
The reason we will require two functions $f_i$, $i=0,1$ arises from the observations earlier that independent sets in $\mathcal{R}_d(G)$ (resp. $\mathcal{R}_d^{lc}(G)$) contains at most $d|V(F)|-\binom{d+1}{2}$ edges (resp. at most $d|V(F)|$ edges).
However when $d=2t-1$, for a set $X\subset V$ with $|X|=2t+1=d+2$ we have
$$d|X|-{d+1\choose 2}+1={d+2\choose 2}={2t+1\choose 2}=t(2t+1)=t|X|.$$
That is, in this special case the sparsity count for $\mathcal{R}_d^{lc}(G)$ does not imply the corresponding sparsity count for $\mathcal{R}_d(G)$. However, this is not the case when $d\geq 2t$. In words perhaps more familiar to rigidity theorists, the edge set of $K_{d+2}$ is a circuit in $\mathcal{R}_d(G)$ while having at most $d|V(F)|$ edges. But such phenomena does not occur when $d\geq 2t$ by Theorem \ref{thm:d-dimchar}. Thus we need to use $f_1$ when $d=2t-1$ and $f_0$ when $d\geq 2t$.
\end{rem}

We shall obtain matroids from $f_i$, $i=0,1$ using Theorem \ref{thm:polymatroid}. It is easy to see that $(E\cup L, f_0)$ is a polymatroid and hence $f_0$ induces a matroid by Theorem \ref{thm:polymatroid}. However, $(E\cup L, f_1)$ does not form a polymatroid so we cannot directly get a matroid from $f_1$. The main reason for this is $f_1$ is not (intersecting) submodular.\footnote{A function is called {\em intersecting submodular} if the submodularity condition holds for pairs of sets with non-empty intersection. If $f_1$ were intersecting submodular
we could use Edmonds' result \cite{Edm} in order to get a matroid. See, for example \cite[Theorem~13.4.2]{AF} for the statement of Edmonds' theorem specifically for intersecting submodular functions.}
While each of the three pieces of $f_1$ is submodular, it is not difficult to see that submodularity fails when we take sets $X,Y$ for which different pieces of $f_1$ apply. To see that intersecting submodularity fails take a graph in which $A$ is the edge set of a subgraph isomorphic to $K_{2t+1}$ and let $B$ be the edge set of a subgraph isomorphic to $P_3$ such that one of its edges is in $A$ and the other is not in $A$. Then $A\cap B\neq \emptyset$ and we have $$f_1(A)+f_1(B)=(2t^2+t-1)+2=2t^2+t+1 < (2t^2+2t)+1=f_1(A\cup B)+f_1(A\cap B)$$
since $t$ is a positive integer.
However, for pairs of sets with sufficiently large intersections $f_1$ behaves like a submodular function. In order to deal with small sets, we will use the Dilworth truncation $f_1^D$ of $f_1$ and turn it into a submodular function. To this end, let us define $f_1^D:2^{E\cup L}\rightarrow \bZ$,
$$
f_1^D(T):=\min\big\{\sum_{j=1}^kf_1(T_j):\{T_1\ldots,T_k\} \text{ is a partition of }T\big\}.
$$
First we prove a useful lemma.
\begin{lem}\label{lem:submodular_intersecting_large}
Suppose $T_1,T_2\subseteq E\cup L$ are sets, each of which spans at least $2t+1$ vertices or contains a loop and $T_1\cap T_2\neq \emptyset$.  Then $f_1(T_1)+f_1(T_2)\geq f_1(T_1\cup T_2)+f_1(T_1\cap T_2)$.
\end{lem}
\begin{proof}
We may write $f_1(T_1)=t|V(T_1)|-a$ and $f_1(T_2)=t|V(T_2)|-b$ with $a,b\in\{0,1\}$. By symmetry, may assume $a\leq b$. Then keeping $t>1$ in mind, we obtain
\begin{align*}
    f_1(T_1)+f_1(T_2)=& t|V(T_1)|-a+t|V(T_2)|-b\\
                     =& t|V(T_1)\cup V(T_2)|-a+t|V(T_1)\cap V(T_2)|-b\\
                  \geq& t|V(T_1\cup T_2)|-a+t|V(T_1\cap T_2)|-b\\
                  \geq& f_1(T_1\cup T_2)+f_1(T_1\cap T_2)
\end{align*}
where the penultimate inequality follows from the fact that $|V(T_1)\cap V(T_2)|\geq |V(T_1\cap T_2)|$. The final inequality follows trivially when $b=0$, and for $b=1$ it follows from the facts that $t>1$ and $T_1\cap T_2\neq \emptyset$.
\end{proof}
Dunstan \cite{Dunstan} showed that the Dilworth truncation $f^D$ of a set function $f$ gives a submodular function provided that $f$ is (intersecting) submodular. The next lemma states that $f_1^D$ is submodular.
Although our statement can be obtained via Dunstan's original proof technique our function $f_1$ is not (intersecting) submodular so we include a full proof for the sake of completeness.

\begin{lem}\label{lem:submodular-f1}
The function $f_1^D$ defined above is submodular.
\end{lem}
\begin{proof}
Let $A,B\subseteq E\cup L$. Let 
$$
f_1^D(A)=\sum_{j=1}^mf_1(A_j)
\qquad \mbox{ and } \qquad
f_1^D(B)=\sum_{k=1}^nf_1(B_n)
$$
where $\mathcal{A}=\{A_j:1\leq j\leq m\}$ is a partition of $A$, and $\mathcal{B}=\{B_k:1\leq k\leq n\}$ is a partition of $B$. Since for a set $T\subseteq E$ with $|V(T)|\leq 2t$, $f_1(T)=|T|=\sum_{e\in T}|\{e\}|=\sum_{e\in T}f_1(\{e\})$, by partitioning into singletons, we may assume that each $A_j,B_k\subseteq E$ has size 1 or spans at least $2t+1$ vertices. 

Construct a bipartite graph $\mathcal{G}=(\mathcal{V},\mathcal{E})$ with the bipartition $\mathcal{V}=\mathcal{A}\cup \mathcal{B}$ and $A_jB_k\in\E$ if $A_j\cap B_k\neq \emptyset$. Consider a connected component $\C$ of $\G$. Let $V(\C)=\{C_1,C_2,\ldots,C_s\}$ be an ordering such that $(\cup_{j=1}^{b-1}C_j)\cap C_b\neq\emptyset$ for all $1\leq b\leq s$ that is $\G[\cup_{j=1}^b C_j]$ is connected for all $1\leq b\leq s$.
Note that since each $C_i$ is either $A_j$ or $B_k$ for some $j$ or $k$, and each $C_i$ contains either a simple edge or spans at lest $2t+1$ vertices.
Then by repeated applications of Lemma \ref{lem:submodular_intersecting_large} and the fact that $f_1(A_j)+f_1(B_k)=f_1(A_j\cup B_k)+f_1(A_j\cap B_k)$ when $|A_j|=1$ (resp.\ $|B_k|=1$) together with $A_j\subseteq B_k$ (resp.\ $B_k\subseteq A_k$) we can write
\begin{align*}
\sum_{j=1}^sf_1(C_j)&=f_1(C_1)+f_1(C_2)+\cdots +f_1(C_s)\\
    &\geq f_1(C_1\cup C_2)+f_1(C_1\cap C_2)+f_1(C_3)+\cdots+f_1(C_s)\\
    &\geq f_1(C_1\cup C_2\cup C_3)+f_1\big((C_1\cup C_2)\cap C_3\big)+f_1(C_1\cup C_2)+f_1(C_4)+\cdots+f_1(C_s)\\
    &\ \vdots\\
    &\geq f_1\big(C_1\cup C_2\cup\ldots\cup C_s\big)+\sum_{b=2}^sf_1\big((C_1\cup\ldots\cup C_{b-1})\cap C_b\big).
\end{align*}
Now, applying the same argument to all connected components $\C$ of $\G$ we can get
\begin{align*}
    f_1^D(A)+f_1^D(B)&=\sum_{j=1}^mf_1(A_j)+\sum_{k=1}^mf_1(B_k)\\
    &= \sum_{\substack{\C: \text{ component of }\G\\ V(\C)=\{C_1,\ldots, C_s\}}}\bigg(\sum_{j=1}^sf_1(C_j) \bigg)\\
    &\geq \sum_{\substack{\C: \text{ component of }\G\\ V(\C)=\{C_1,\ldots, C_s\}}} \bigg(f_1(C_1\cup\ldots \cup C_s)+\sum_{b=2}^sf_1\big((C_1\cup\ldots\cup C_{b-1})\cap C_b\big) \bigg)\\
    &=\sum_{\substack{\C: \text{ component of }\G\\ V(\C)=\{C_1,\ldots, C_s\}}} f_1(C_1\cup\ldots \cup C_s) \\
    &\ \ \ \ \ + \sum_{\substack{\C: \text{ component of }\G\\ V(\C)=\{C_1,\ldots, C_s\}}} \sum_{b=2}^sf_1\big((C_1\cup\ldots\cup C_{b-1})\cap C_b\big) \bigg)\\
    &\geq f_1^D(A\cup B)+f_1^D(A\cap B),
\end{align*}
where the last inequality follows since the family 
$$
\{C_1\cup \ldots \cup C_s: V(\C)=\{C_1,C_2,\ldots ,C_s\}, \C \text{ is a connected component of }\G\}
$$
is a partition of $A\cup B$, and the family 
$$
\{(C_1\cup\ldots \cup C_{b-1})\cap C_b: V(\C)=\{C_1,\ldots ,C_s\}, \C \text{ is a connected component of }\G, 2\leq b\leq s\}
$$
is a partition of $A\cap B$.
\end{proof}

It is straightforward to verify that $f_1^D$ is monotone and normalized. Thus it follows from Lemma \ref{lem:submodular-f1} that $(E\cup L,f_1^D)$ is a polymatroid. This leads us to the next theorem. The $i=1$ part of this theorem can be considered a special case (since the function $f_1$ is fixed) of Edmonds' theorem in \cite{Edm}. Edmonds gave the statement and its proof, and Antolini, Dewar and Tanigawa \cite{ADT} very recently adapted Edmonds' statement and proof for their own purposes. Although our statement and proof are almost identical to that of \cite{ADT} with $f_1$ being a fixed function, we present the proof here for the sake of completeness since they assume $f_1$ to be submodular.
\begin{thm}\label{thm:edmonds-both}
Let $G=(V,E,L)$ be a looped simple graph $f_1^D$ and $f_i$, for $i=0,1$, be the functions defined above. Put
$$
\I_{f_i}:=\{T\subseteq E\cup L: |I|\leq f_i(I) \text{ for all } I\subseteq T\}.
$$
Then $(E\cup L,\I_{f_i})$ is a matroid with rank function $\hat f_i:2^{E\cup L}\to \bZ$ given by
$$
\hat f_0(T) :=\min\{|T'|+f_0(T\setminus T'):T'\subseteq T\}
$$
and
\begin{align*}
\hat f_1(T):&=\min\{|T'|+f_1^D(T\setminus T'):T'\subseteq T\}\\
&=\min\big\{|T'|+\sum_{j=0}^kf_1(T_j):T'\subseteq T \text{ and } \{T_0,T_1,\ldots,T_k\} \text{ is a partition of }T\setminus T'\big\}
\end{align*}
respectively for $i=0$ and $i=1$.
\end{thm}
\begin{proof}
First suppose $i=0$. Since $(E\cup L, f_0)$ is a polymatroid, Theorem \ref{thm:polymatroid} implies that the matroid $(E\cup L,\I_{f_0})$ is exactly the matroid induced by $f_0$ and its rank function is $\hat f_0$.

Now suppose $i=1$. Since $(E\cup L,f_1^D)$ is a polymatroid, by Theorem \ref{thm:polymatroid} we obtain a matroid whose independent sets are
$$
\I_{f_1^D}=\{T\subseteq E\cup L: |I|\leq f_1^D(T) \text{ for all } I\subseteq T\}.
$$
We claim that $\I_{f_1^D}=\I_{f_1}$. The fact that $f_1^D(I)\leq f_1(I)$ for all $I\subseteq E\cup L$ gives $\I_{f_1^D}\subseteq \I_{f_1}$. Let $T\in \I_{f_1}$. Then all subsets of $T$ are contained in $\I_{f_1}$ by the definition of $\I_{f_1}$. Take a subset $I\subseteq T$ and let $\{I_1,I_2,\ldots,I_k\}$ be a partition of $I$ from which $f_1^D(I)$ can be obtained. Then $|I_j|\leq f_1(I_j)$ holds since $I_j\in \I_{f_1}$ and we can write
$$
|I|=\sum_{j=1}^k|I_j|\leq \sum_{j=1}^kf_1(I_j)=f_1^D(I)
$$
which implies $T\in \I_{f_1^D}$.
The rank function for $(E\cup L, \I_{f_1})$ now follows from Theorem \ref{thm:polymatroid} and the definition of $f_1^D$.
\end{proof}

Given a matroid $\M=(E,\I)$ and a subset $S\subset E$, let $\M/S$ denote the matroid obtained from $\M$ by contracting the set $S$. By Theorems \ref{thm:d-dimchar} and \ref{thm:edmonds-both} we have the following.

\begin{cor}\label{cor:isomorphic-both}
Let $G=(V,E,L)$ be a looped simple graph and $d,t$ be integers with $d\geq 2$ and $d\geq 2t-1$. Then
\begin{itemize}
    \item [(i)] when $d= 2t-1$, the matroid $(E\cup L, \I_{f_1})$; and
    \item[(ii)] when $d\geq 2t$, the matroid $(E\cup L, \I_{f_0})$ 
\end{itemize}
is isomorphic to $\calR^{lc}_d(G^{[d-t]})/L'$ where $L'$ is the set of the loops contained in $G^{[d-t]}-G$. Thus

\begin{align*}
r^{lc}_d(T\cup L')=&|L'|+\hat f_1(T)\\
		   =&(d-t)|V|+\hat f_1(T)\\
		   =&(d-t)|V|+\min\big\{|T'|+\sum_{j=0}^kf_1(T_j):T'\subseteq T, \{T_0,\ldots,T_k\} \text{ is a partition of }T\setminus T'\big\}
\end{align*}
holds when $d=2t-1$ and 
$$
r_d^{lc}(T\cup L') = |L'|+\hat f_0(T)=(d-t)|V|+\min\{|T'|+f_0(T\setminus T')\text{ for all }T'\subseteq T\}
$$
holds when $d\geq 2t$.
\end{cor}
\begin{proof}
The statement follows by comparing the independent sets in each matroid.
\end{proof}

Let $G=(V,E,L)$ be a looped simple graph, $T'\subseteq T\subseteq E\cup L$ and $\{T_0,T_1,\ldots,T_k\}$ be a partition of $T\setminus T'$. Suppose
$|V(T_j)|\leq 2t$ for some $t>1$, and $T_j\cap L=\emptyset$ for some $0\leq j\leq k$. Then we have $|T_j|= f_1(T_j)$ which in turn gives
$$
|T'\cup T_j|+\sum_{s\in\{0,1,...,k\}\setminus\{j\}}f_1(T_s)= |T'|+\sum_{n=0}^kf_1(T_n).
$$
Combining this with Corollary \ref{cor:isomorphic-both} we obtain the following.
\begin{lem}\label{lem:geq2t+1}
Let $G=(V,E,L)$ be a looped simple graph and $L'$ denote the loops contained in $G^{[d-t]}-G$ for some positive integers $d,t$ with $d\geq 2$ and $d= 2t-1$. Then
$$
r^{lc}_d(T\cup L')=(d-t)|V|+\min\big\{|T'|+\sum_{j=0}^kf_1(T_j):T'\subseteq T \text{ and } \{T_0,\ldots,T_k\} \text{ is a partition of }T\setminus T'\big\}
$$
where $|V(T_j)|\geq 2t+1$ for all $T_j$ with $T_j\cap L=\emptyset$, $0\leq j\leq k$.
\end{lem}

\begin{lem}\label{lem:union-both}
Let $G=(V,E,L)$ be a looped simple graph $T_1,T_2\subseteq E\cup L$ be sets, $t>1$ be an integer. Suppose either
\begin{itemize}
    \item[(i)] $T_j\cap L\neq \emptyset$ or $|V(T_j)|\geq 2t+2$ for each $j=1,2$; or
    \item[(ii)] $|V(T_1)\cap V(T_2)|\geq 1$ and $|V(T_j)|= 2t+1$ for each $j=1,2$
\end{itemize} 
holds. Then we have $f_1(T_1\cup T_2)\leq f_1(T_1)+f_1(T_2)$.
\end{lem}
\begin{proof}
Let $f_1(T_1)=t|V(T_1)|-a$ and $f_1(T_2)=t|V(T_2)|-b$ where $a,b\in\{0,1\}$.
\begin{align*}
    f_1(T_1)+f_1(T_2)&\geq t|V(T_1)|-a+t|V(T_2)|-b\\
                     &=t|V(T_1)\cup V(T_2)|-a+t|V(T_1)\cap V(T_2)|-b\\
                     &\geq f_1(T_1\cup T_2)
\end{align*}
where the last inequality holds trivially when (i) applies as this forces $a=b=0$ and it holds when (ii) applies as this forces $t|V(T_1)\cap V(T_2)|\geq a+b$.
\end{proof}

\begin{lem}\label{lem:edge_char-both} Let $G=(V,E,L)$ be a looped simple graph, $d,t$ be positive integers with $d\geq 2$ and $d\geq 2t-1$ and $f_i$ be the function defined above for $i\in \{0,1\}$. Then 
\begin{itemize}
\item[(i)] when $d=2t-1$
$$
r^{lc}_d(T\cup L')=(d-t)|V|+\min\big\{|T'|+\sum_{j=0}^kf_1(T_j):T'\subseteq T \text{ and } \{T_0,\ldots,T_k\} \text{ partitions }T\setminus T'\big\}
$$
where the partitions allowed satisfy the following: only $T_0$ may contain loops and if it does not then $|V(T_0)|\geq 2t+2$; $|V(T_j)|= 2t+1$ for all $1\leq j\leq k$; $V(T_j)\cap V(T_l)=\emptyset$ for all distinct $j$ and $l$. 
\item[(ii)] when $d\geq 2t$
$$
r^{lc}_d(T\cup L')=(d-t)|V|+\min\big\{|T'|+f_0(T_0):T_0=T\setminus T'\}
$$
where $|V(T_0)|\geq 2t+2$ or $T_0$ contains some loops.

\end{itemize}
\end{lem}
\begin{proof}
When $d=2t-1$ (since $d\geq 2$, this forces $t\geq 2$) combining Corollary \ref{cor:isomorphic-both} and Lemmas \ref{lem:geq2t+1} and \ref{lem:union-both} we may assume that $T_0$ is the only member of the partition that may contain a loop or spans at least $2t+2$ vertices. Similarly, we may also assume $|V(T_j)|=2t+1$ for all $1\leq j\leq k$ and $V(T_j)\cap V(T_l)=\emptyset$ by Lemmas \ref{lem:geq2t+1} and \ref{lem:union-both}.

For $d\geq 2t$ assume $T_0\subseteq E$ and $|V(T_0)|\leq 2t+1$. This implies $|T_0|\leq {|V(T_0)|\choose 2}\leq t|V(T_0)|$ and we obtain
$$
|T| + f_0(\emptyset) = |T'|+|T_0|\leq |T'|+t|V(T_0)|=|T'|+f_0(T_0).  
$$
This implies that we can add edges from $T_0$ to $T'$ to obtain a smaller value.
\end{proof}

We can now derive the main result of this section. This gives a reformulation of Lemma \ref{lem:edge_char-both} describing the rank function of $\mathcal{R}_d^{lc}$ in terms of covers. To this end we need to give some definitions for the case when $d=2t-1$.

Let $G=(V,E,L)$ be a looped simple graph and $t\geq 2$ be an integer.
A family $\X=\{X_0,X_1,\ldots,X_k\}$ of subsets of $V$ is said to be a {\em cover} of $G$ if every edge $e\in E$ and every loop $l\in L$ is induced by some member
of $\X$. A cover $\X=\{X_0,X_1,\ldots,X_k\}$ of $G-T$ for some $T\subset E\cup L$ is called {\em admissible} if every loop in $G-T$ is induced by the vertices
in $X_0$, every loop in $T$ is induced by the vertices in $V\setminus X_0$ and $|X_j|=  2t+1$ for all $1\leq j\leq k$.
The set $X_0$, which could possibly be empty, is called the {\em looped member} of the cover.
The {\em value} of the admissible cover (or just cover when $k=0$) $\X$ of $G-T$ is defined to be
$$
|T|+t|X_0|+\sum_{j=1}^k(t|X_j|-1)=|T|+t|X_0|+k(2t^2+t-1)
$$
and denoted by $\val_t(\X)$.
The family $\X$ is called {\em non-intersecting} if $X_j\cap X_l=\emptyset$ for all distinct $j$ and $l$.

Let $G=(V,E,L)$ be a looped simple graph with $X\subseteq V$. We use $E_G(X)$ and $L_G(X)$ to denote the set of simple edges and loops respectively in $G[X]$.

\begin{thm}\label{thm:rank_cover-both}
Let $G=(V,E,L)$ be a looped simple graph, $d\geq 2$ and $d\geq 2t-1$. Then
$$
    r^{lc}_d(G^{[d-t]})=(d-t)|V|+\min \val_t(\X)
$$
where, when $d=2t+1$ the minimum is taken over all $T\subseteq E\cup L$ and all non-intersecting admissible covers $\X=\{X_0,X_1,\ldots,X_k\}$ of $G-T$ and, when $d\geq 2t$
the minimum is taken over all $T\subseteq E\cup L$ and all covers $\X=\{X_0\}$ (i.e., $k=0$) of $G-T$.
\end{thm}
\begin{proof}
Let $\X=\{X_0,X_1,\ldots,X_k\}$ be a non-intersecting admissible cover of $G-T$ for some $T\subseteq E\cup L$ where $k=0$ when $d\geq 2t$. Consider $T_0:=E_G(X_0)\cup L_G(X_0)$; $T_j:=E_G(X_j)$,
$1\leq j\leq k$. Clearly $V(T_j)\subseteq X_j$ for all $0\leq j\leq k$. For some $T_j$, $0\leq j\leq k$ we may have $T_j=\emptyset$ and there may be some $T_j$ with $|V(T_j)|\leq 2t+1-i$ and
$T_j\cap L=\emptyset$ where $i=1$ when $d=2t-1$ and $i=0$ when $d\geq 2t$. By replacing all such $T_j$ by
$$
T':=\bigcup\limits_{\substack{T_j:|V(T_j)|\leq 2t+1-i \\ \ T_j\cap L=\emptyset}}T_j
$$
we may assume that $\{T_0,T_1,\ldots,T_m\}$ is a partition of $(E\cup L)\setminus (T\cup T')$ for some $0\leq m\leq k$, and the sets $T_j$, $0\leq j\leq m$ satisfy
the conditions in Lemma \ref{lem:edge_char-both} when $d=2t-1$ together with $i=1$ as well as when $d\geq 2t$ together with $i=0$. Note that we have $f_i(T_0)\leq t|X_0|$ for all $i\in\{0,1\}$ (since $V(T_0)\subseteq X_0$) and $f_1(T_j)=t|X_j|-1$ for all $1\leq j\leq m$
(since $|V(T_j)|=|X_j|=2t+1$ for all $1\leq j\leq m$). We can now use Lemma \ref{lem:edge_char-both} to obtain
\begin{align*}
r_d^{lc}(G^{[d-t]})&\leq (d-t)|V| + |T\cup T'|+\sum_{j=0}^mf_i(T_j)\\
				 &\leq (d-t)|V| + |T|+\sum_{j=0}^kf_i(T_j)\\
				 &\leq (d-t)|V|+|T|+t|X_0|+\sum_{j=1}^k(2|X_j|-1)\\
				 &=(d-t)|V|+\val_t(\X),
\end{align*}
since we have $i=1$ when $d=2t-1$ and $k=m=i=0$ when $d\geq 2t$, $|T_j|\leq f_i(T_j)$ for every $T_j$ used in the definition of $T'$ and $V(T_j)\subseteq X_j$ for all $0\leq j\leq k$.

It remains to show that there exists a non-intersecting admissible cover $\X=\{X_0,X_1,\ldots,X_k\}$ for which $\displaystyle{r_d^{lc}(G^{[d-t]})=(d-t)|V|+\val_t(\X)}$ where $k=0$ when $d\geq 2t$. To prove this, let
$T\subseteq E\cup L$ and $T_0,T_1,\ldots,T_k$ be a partition of $(E\cup L)\setminus T$ satisfying the conditions in the statement of Lemma \ref{lem:edge_char-both} from which
$\displaystyle{r_d^{lc}(G^{[d-t]})}$ can be obtained, that is
$$
r_d^{lc}(G^{[d-t]})=(d-t)|V|+|T|+\sum_{j=0}^kf_i(T_j)
$$
where $i=1$ when $d=2t-1$ and $k=i=0$ when $d\geq 2t$.
We split the proof into two cases.\\

\noindent\textbf{Case 1.} $T_0\cap L\neq \emptyset$ or $|V(T_0)|\geq 2t+2$.\\

Consider $\X=\{X_0,X_1,\ldots,X_k\}$ where $X_j=V(T_j)$ for all $0\leq j\leq k$. We claim that every loop in $T$ is induced by some $X_j$, $1\leq j\leq k$.
To see this let us suppose the contrary and $l\in T$ be a loop incident with a vertex in $X_0$. We obtain that
$T_0\cup\{l\},T_1,\ldots,T_k$ is a partition of $(E\cup L)\setminus (T\setminus \{l\})$. Now Lemma \ref{lem:edge_char-both} and the fact that $f_i(T_0)=f(T_0\cup\{l\})$ 
imply that
\begin{align*}
r_d^{lc}(G^{[d-t]}) &=(d-t)|V|+|T\setminus\{l\}|+f_i(T_0\cup\{l\})+\sum_{j=1}^kf_i(T_j) \\  
                    &=(d-t)|V|+|T|+\sum_{j=0}^kf_i(T_j)-1\\
                    &=r_d^{lc}(G^{[d-t]})-1,
\end{align*}
a contradiction. Thus, by the choice of $T_j$, $1\leq j\leq k$, the family $\X$ is a non-intersecting admissible cover of $G-T$ when $d=2t-1$ (together with $i=1$) and $\X=\{X_0\}$, that is, $k=0$ is a cover of $G-T$ when $d\geq 2t$ (together with $i=0$).
By the assumption of Case 1 and the choice of $T_j$, $0\leq j\leq k$, we have $f_i(T_0)=t|X_0|$ and $f_i(T_j)=t|X_j|-1$, $1\leq j\leq k$ when $k\geq 1$.
This gives
\begin{align*}
    r_d^{lc}(G^{[d-t]})&=(d-t)|V|+|T|+\sum_{j=0}^kf_i(T_j)\\
                       &=(d-t)|V|+|T|+t|X_0|+\sum_{j=1}^k(t|X_j|-1)\\
                       &=(d-t)|V|+\val_t(\X),
\end{align*}
as desired, where $i=1$ when $d=2t-1$ and $i=k=0$ when $d\geq 2t$.\\

\noindent\textbf{Case 2.} $T_0\subset E$ and $|V(T_0)|\leq 2t+1$.\\

Since Lemma \ref{lem:edge_char-both} (ii) requires either $|V(T_0)|\geq 2t+2$ or $T_0\cap L\neq \emptyset$, we may deduce that Lemma \ref{lem:edge_char-both} (i) applies and so we have $d=2t-1$.
Consider $\X=\{X_0,X_1,\ldots,X_k,X_{k+1}\}$ where $X_0=\emptyset$, $X_j=V(T_j)$ for all $1\leq j\leq k$ and $X_{k+1}=V(T_0)$.
Then $\X$ is a non-intersecting admissible cover of $G-T$ whose looped member is $\emptyset$. We have $f_1(T_0)=t|X_{k+1}|-1$, $f_1(T_j)=t|X_j|-1$ for all
$1\leq j\leq k$. Thus we can deduce that
\begin{align*}
r_d^{lc}(G^{[d-t]})=(d-t)|V|+|T|+\sum_{j=0}^kf_1(T_j)&=(d-t)|V|+|T|+t|\emptyset|+\sum_{j=1}^{k+1}f_1(X_j)\\
										    &=(d-t)|V|+|T|+t|X_0|+\sum_{j=1}^{k+1}(t|X_j|-1)\\
											&=(d-t)|V|+\val_t(\X)
\end{align*}
as desired.
\end{proof}

\section{$k$-Balanced Graphs and Rigidity}\label{sec:d-dim_main}
In this section, we will prove Theorem \ref{thm:d_geq2_main} giving a $d$-dimensional sufficient condition for a looped simple graph $G=(V,E,L)$ to be $\mathcal{L}_d$-rigid.
Recall that $G$ is {\em $k$-balanced} if every connected
component of $G-T$, where $T\subseteq V$ with $|T|\leq k$, has at least $k-|T|$ vertices with loops. Recall also that $G$ is {\em weakly $k$-balanced} if every connected
component of $G-T$, where $T\subseteq V$ with $|T|\leq k$, has at least $k-|T|$ loops. 
See Figure \ref{fig:3balanced-4weakbalanced} and consider the looped simple graphs $G$ and $H$ in this figure. Clearly both are 3-balanced, and none of them is 4-balanced. Note that the disjoint union of $G$ and $H$ is also 3-balanced. For the weak version of balancedness, since both are 3-balanced, both are weakly 3-balanced. However, $H$ is also weakly 4-balanced, but $G$ is not. We can also deduce that their disjoint union is weakly 3-balanced.
\begin{figure}[hbt]
\begin{center}

\begin{tikzpicture}[scale=1,font=\small]
%\fill[pattern=north west lines,opacity=.6,draw] (0,0) circle [radius=4];
\begin{scope}[rotate=90]

\foreach \i in {1,2,3}
{
	\node[roundnode] (\i) at ({(\i-1)*120}: 1cm)  [] {};
}
\foreach \i/\j in {1/2,2/3,3/1}
{
	\draw[-] (\i) to (\j);
	\pgfmathsetmacro\startp{(\i-1)*120-20}
	\pgfmathsetmacro\endp{(\i-1)*120+20}
	\draw[-, in=\startp, out=\endp,loop] (\i) to (\i);
}
\end{scope}
\node [label={[label distance=0.5cm]270:$G$}] {};

\begin{scope}[xshift=3cm,rotate=90]
\foreach \i in {1,2,3}
{
	\node[roundnode] (\i) at ({(\i-1)*120}: 1cm)  [] {};
}
\foreach \i/\j in {1/2,2/3,3/1}
{
	\draw[-] (\i) to (\j);
	\pgfmathsetmacro\startp{(\i-1)*120-40}
	\pgfmathsetmacro\endp{(\i-1)*120}
	\draw[-, in=\startp, out=\endp,loop] (\i) to (\i);
    \pgfmathsetmacro\startp{(\i-1)*120}
	\pgfmathsetmacro\endp{(\i-1)*120+40}
	\draw[-, in=\startp, out=\endp,loop] (\i) to (\i);
} 
\node [label={[label distance=0.5cm]270:$H$}] {};
\end{scope}

%\node[roundnode] at (-0.7cm, 0cm) (u) [] {};
%\node[roundnode] at (0.7cm, 0cm) (v) [] {};
%\node[roundnode] at (0cm, 0.7cm) (a) [] {}
%	edge[] (u)
%	edge[] (v);
\end{tikzpicture}
\end{center}
\caption{Balancedness and weak balancedness.}
\label{fig:3balanced-4weakbalanced}
\end{figure}
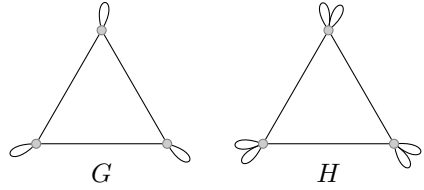

\begin{proof}[Proof of Theorem \ref{thm:d_geq2_main}]
    Fix $d$ and $t$ with $d\geq 2t-1$. Let $G=(V,E,L)$, $F$ be a pair representing a counterexample with $|V|+|L|+|F|$ as small as possible and with respect to this let $|E|$ be as large as possible.
    Let $H=G-F$ and so $H^{[d-t]}=G^{[d-t]}-F$. For the case $d=2t-1$ let $T\subseteq E(H)\cup L(H)$ and $\X=\{X_0,X_1,\ldots,X_k\}$ be a non-intersecting admissible cover, and for other possibilities of $d$ and $t$ let $T\subseteq E(H)\cup L(H)$ and $\X=\{X_0\}$ be the cover from which $r_d^{lc}(H^{[d-t]})=r_d^{lc}(G^{[d-t]}-F)$ can be obtained (using Theorem \ref{thm:rank_cover-both}). That is,
\begin{align*}
    (d-t)|V|+|T|+t|X_0|+k(2t^2+t-1)&=(d-t)|V|+|T|+t|X_0|+\sum_{i=1}^k(t|X_i|-1)\\            
                       &=(d-t)|V|+\val_t(\X)\\
                       &=r_d^{lc}(H^{[d-t]})=r_d^{lc}(G^{[d-t]}-F)\\
                       &<d|V|,
\end{align*}
where $k=0$ when $d\geq 2t$.

Set $T'=T\cup F$. Note that the minimality of $|V|+|L|+|F|$ implies that no edge or loop in $T'$ is induced by the vertices in $X_0$. Consider the family $\{V\}\cup\{\{f\}:f\in T'\}\cup \X$.
We will now use the discharging method on this family in order to get a contradiction. Let $\mu(x)$ denote the initial charge of $x$ and set the initial charges on this family as $\mu(V)=(d-t)|V|$, $\mu(f):=\mu(\{f\})=1$ for all $f\in T'$, $\mu(X_0)=t|X_0|$ for all possible choices of $d,t$ and let $\mu(X_i)=2t^2+t-1$ for $1\leq i\leq k$ when $d=2t-1$. Note that $k=0$ when $d\geq 2t$. Letting $\mu_{\text{total}}$ be the total amount of charge and using $T'=T\cup F$ with $|F|\leq t$ as well as the inequality above, we can write
\begin{equation}\label{eq:total_charge}
    \mu_{\text{total}}=\mu(V)+\sum_{f\in T'}\mu(f)+\sum_{i=0}^k\mu(X_i)=(d-t)|V|+|T'|+t|X_0|+k(2t^2+t-1)<d|V|+t
\end{equation}
for all choices of $d,t$.
For the reallocation of charges we will get some of the charges from the sets and give them to the vertices carefully. We will use the notation $\sigma_x(Y)$ to denote the charge $x$ gets from $Y$.
The rules we will follow are as follows:
\begin{itemize}
\item R1 - For each $v\in V$, $\sigma_v(V)=d-t$,
\item R2 - For each $f\in T'$ and each $x$ incident with $f$:
\begin{itemize}
    \item R2A - If $x\in X_0$ $\sigma_x(f)=0$,
    \item R2B - If $x\notin X_0$, $\sigma_x(f)=\frac{1}{2}$
\end{itemize}
\item R3 - For each $x\in X_0$, $\sigma_x(X_0)=t$
\item R4 - When $d=2t-1$, for each $X_i$, $1\leq i\leq k$, $\sigma_x(X_i)=t-\frac{s}{2}$ where $s$ is the number of edges or loops in $T'$ incident with $x$.
\end{itemize}
Let $\mu'(x)$ denote the final charge of $x$, i.e., the charge after the reallocation process according to the rules above.

\begin{claim}\label{claim:mu}
    $\mu'(v)\geq d$ for all $v\in V$.
\end{claim}

\begin{proof}[Proof of Claim \ref{claim:mu}]
    First suppose $v\in X_0$. Then $\mu'(v)=(d-t)+t=d$ after $v$ gets $d-t$ from $V$ by (R1) and $t$ from $X_0$ by (R3).
    Next suppose $v\in X_i$ for some $1\leq i\leq k$. Then we need to have $d=2t-1$. Let $s$ denote the number of edges or loops in $T'$ that are incident with $v$. Thus we have
    $\mu'(v)= (d-t)+s\cdot\frac{1}{2}+t-\frac{s}{2}=d$, after $v$ gets $(d-t)$ from $V$ by (R1), $\frac{1}{2}$ from each member in $T'$ incident with $v$ by (R2B) and $t-\frac{s}{2}$ from $X_i$ by (R4). 

    Finally suppose $v$ is not contained in any $X_i$, $0\leq i\leq k$. This forces that all edges and loops incident with $v$ are contained in $T'$. Since $G$ is weakly $2t$-balanced, there are at least $2t$ such members in $T'$. This gives $\mu'(v)\geq (d-t)+2t\cdot \frac{1}{2}=d$ after $v$ gets $(d-t)$ from $V$ by (R1) and $\frac{1}{2}$ from each incident member in $T'$ by (R2B).
\end{proof}

\begin{claim}\label{claim:mu2}
    $\mu'(X_i)\geq 0$ for all $0\leq i\leq k$.
\end{claim}

\begin{proof}[Proof of Claim \ref{claim:mu2}]
    For $i=0$, we have $\mu'(X_0)=\mu(X_0)-t|X_0|=0$ since the starting charge $\mu(X_0)$ of $X_0$ is $t|X_0|$ and $X_0$ gives $t$ to each of its members by (R3).
    For $1\leq i\leq k$ we need to have $d=2t-1$ and $|X_i|=2t+1$ and so $\mu(X_i)=2t^2+t-1$.
    Then since $G$ is weakly $2t$-balanced, there are at least $2t$ members in $T'$ which are incident with vertices in $X_i$ as $\X$ is non-intersecting. This gives $$
    \mu'(X_i)\geq \mu(X_i)-\big((2t+1)\cdot t-2t\cdot\frac{1}{2}\big)=2t^2+t-1-2t^2+t-t=t-1\geq 0
    $$ since $t$ is a positive integer.
\end{proof}

\begin{claim}\label{cla:mu'_f}
We have
\begin{itemize}
\item[(a)] $\mu'(f)= 0$ for all $f\in (T'\cap E)$ such that $f$ is not incident with a vertex in $X_0$,
\item[(b)] $\mu'(f)=\frac{1}{2}$ for all $f\in(T'\cap L)$ and for all $f\in (T'\cap E)$ such that $f$ is incident with a vertex in $X_0$.
\end{itemize}
\end{claim}

\begin{proof}[Proof of Claim \ref{cla:mu'_f}]
    Follows from (R2A-B) as a loop has only one incident vertex and a simple edge has two of them.
\end{proof}

\noindent\textbf{Case 1.} $T'=\emptyset$.\\

Then we have $F=\emptyset$ as $F\subseteq T'$ and may assume $X_i=V$ for some $0\leq i\leq k$, by the minimality of $|V|+|L|+|F|$ and $(V,E)$ is a complete graph by the maximality of $|E|$.
First suppose $|V|<2t$. Then since $G$ is weakly $2t$-balanced, $G$ contains $K_{|V|}^{[2t-|V|]}$ as a spanning subgraph. Suppose $|V|$ is odd and recall that $K_{|V|}$ can be partitioned into $\frac{|V|-1}{2}$ Hamilton cycles (c.f.\ \cite{EL}). Thus the fact that each vertex in $G$ has at least $2t-|V|$ loops implies that $G$ has a spanning $t$-tight subgraph
as 
$$\frac{|V|-1}{2}+(2t-|V|)=t+(t-\frac{|V|+1}{2})\geq t.$$
Moreover, this spanning subgraph does not contain a copy of $K_{d+2}$ because $d+2\geq 2t+1$. Thus $G^{[d-t]}$ has a spanning $d$-tight, $K_{d+2}$-free subgraph such that each vertex of this subgraph has at least $(d-t)\geq \lfloor \frac{d}{2}\rfloor$ loops. Hence $G^{[d-t]}$ is $\mathcal{L}_d$-rigid by Theorem \ref{thm:d-dimchar}. So suppose $|V|$ is even. Recall that $K_{|V|}$ can now be partitioned into $\frac{|V|}{2}-1$ Hamilton cycles and a perfect matching (that we shall not need, cf.\ \cite{EL}). The fact that each vertex has at least $2t-|V|$ loops now implies that
$$
\frac{|V|}{2}-1+(2t-|V|)=t+(t-\frac{|V|}{2}-1)\geq t
$$
which again shows that $G$ has a spanning $t$-tight subgraph. We can now proceed similarly to deduce that $G^{[d-t]}$ is $\mathcal{L}_d$-rigid.

Thus we may assume $|V|\geq 2t$. Since $G$ is weakly $2t$-balanced, it has at least $2t$ vertices with loops. Let $S$ denote a set of vertices in $X_0$ with loops and with $|S|=2t$. Thus $G[S]$ contains $K_{2t}^{[1]}$ as a spanning subgraph.
As above we can use the fact that the edge set of $K_{2t}$ can be partitioned into $t-1$ Hamilton cycles and a perfect matching. By also considering a loop for each vertex in $S$ we deduce that the graph  $G[S]$ has a spanning $t$-tight subgraph which is $K_{d+2}$-free as $d+2\geq 2t+1>2t=|S|$. Thus $(G[S])^{[d-t]}$ has a spanning $d$-tight, $K_{d+2}$-free subgraph such that each vertex of in this subgraph has at least $(d-t)\geq \lfloor \frac{d}{2}\rfloor$ loops. Hence $W=(G[S])^{[d-t]}$ is $\mathcal{L}_d$-rigid by Theorem \ref{thm:d-dimchar}. Since each remaining vertex $v\in V\setminus S$ has $2t$ neighbours in $S$ and $(d-t)$ loops in $G^{[d-t]}$, we can sequentially add each such vertex to $W$ by a $d$-dimensional 0-loop extension operation and preserve $\mathcal{L}_d$-rigidity by Lemma \ref{lem:ddim-kloop}. \\

\noindent\textbf{Case 2.} $T'\neq \emptyset$.\\

Let $Y\subset X_0$ denote the set of vertices whose neighbourhood is contained in $X_0$. Let $Z=X_0\setminus Y$, that is $Z$ is the set of vertices in $X_0$ that have at least one neighbour in $V\setminus X_0$. Let $T_l'$ denote the set of loops contained in $T'$. Since $G$ is weakly $2t$-balanced we have $|Z|+|T_l'|\geq 2t$. Then using Eq (\ref{eq:total_charge}) and setting $k=0$ when $d\geq 2t$, we have
\begin{align*}
    d|V|+t>&\mu_{\text{total}}\\
        =&\sum_{v\in V}\mu'(v)+\mu'(V)+\sum_{f\in T'}\mu'(f)+\sum_{i=0}^k\mu'(X_i)\\
    \geq& d|V|+0+2t\cdot \frac{1}{2}+0+0\\
        =&d|V|+t,
\end{align*}
where the last inequality follows from the claims above and the facts that there are $|Z|$ simple edges in $T'$ incident with vertices in $X_0$ and $|Z|+|T_l'|\geq 2t$, a contradiction.
\end{proof}

\section{Essential Balancedness and 2-rigidity}

Let $G=(V,E,L)$ be a looped simple graph and $x\in V$ be a vertex. We use $l_G(x)$ (or $l(x)$ when the corresponding graph $G$ is clear from the context) to denote {\em the number of loops at} $x$ in $G$.
We now prove Theorem \ref{thm:weakly6balanced} showing that we can use weakly 6-balancedness instead of 6-balancedness in the statement of Theorem \ref{thm:2D-6balanced} and its conclusion still holds.

\begin{proof}[Proof of Theorem \ref{thm:weakly6balanced}]
Let $G=(V,E,L)$ and suppose $G-F$ is a counterexample with $|V|$ being as small as possible, and with respect to this, $|L|$ is as small as possible. If all vertices of $G$ had at most one loop then $G$ would be 6-balanced and we could apply Theorem \ref{thm:2D-6balanced} to deduce that $G-F$
is $\mathcal{L}_2$-rigid. Thus there exists a vertex $v\in V$ with at least two loops $l_1,l_2$. Let $N_G(v)=\{x_1,x_2,\ldots,x_t\}$.
By the minimality of $|L|$, the graph $G_1=G-l_1$ is not weakly 6-balanced. Thus there exists a set $T\subseteq V$ such that some component $C_1$ of $G_1-T$ has less than $6-|T|$ loops. Note that $v\in C_v$ since $E(G)=E(G_1)$ and the only missing member in $G_1$ is a loop at $v$ compared to $G$. Since $C_v$ has at least one loop, namely $l_2$, we have $|T|\leq 4$. This forces $V(C_v)=\{v\}$ as otherwise $G-(T\cup\{v\})$ would have a connected component with less than $6-|T\cup\{v\}|$ loops, contradicting the fact that $G$ is weakly 6-balanced. This in turn implies that $N(v)\subseteq T$ and so $|N_G(v)|\leq 4$. Consider the graph $G'=G-F$. We split the proof into two cases.\\

\noindent\textbf{Case 1.} $v$ has at least two loops, say $l_1,l_2$, in $G'$.\\

Take a generic realisation $(G-F,p,q)$ of $G-F$. Form a graph $H$ from $G$ by removing $v$ and adding a loop $s_i$ to each $x_i\in N_G(v)$.

\begin{claim}\label{cla:H_weakly_6_balanced}
The graph $H$ is weakly 6-balanced. 
\end{claim}

\begin{proof}[Proof of Claim \ref{cla:H_weakly_6_balanced}]
Suppose the contrary and let $U\subseteq V(H)$ be a set for which there exists a connected component $C$ of $H-U$ with less than $6-|U|$ loops. If $V(C)\cap N_G(v)=\emptyset$ then $C$ would be a connected component in $G-U$, a contradiction. Hence $x_i\in V(C)$ for some $1\leq i\leq t$. The fact that $s_i$ is a loop at $x_i$ in $H$ implies $C$ has at least one loop and this forces $|U|\leq 4$. This further forces that the connected component containing $x_i$ in $G-(U\cup\{v\})$ has less than $6-|U\cup\{v\}|$ loops. This contradicts the fact that $G$ is weakly 6-balanced and completes the proof of the claim.
\end{proof}
Set $F_H:=F\cap (E(H)\cup L(H))\cup\{s_i:vx_i \in E(G)\cap F\}$, that is $F_H$ is obtained from $F$ by replacing its edges of the form $vx_i$ (if any) by $s_i$. By the minimality of $|V|$ and Claim \ref{cla:H_weakly_6_balanced}, $H-F_H$ is $\mathcal{L}_2$-rigid. Set $q_H(l)=q(l)$ for all $l\in L(G-F)$ and $q_H(s_i)=p(x_i)-p(v)$ for all $s_i\in L(H-F_H)$, and $p_H=p|_{V\setminus \{v\}}$. Then since $(p,q)$ is generic, $(p_H,q_H)$ is also generic, implying that $(H-F_H,p_H,q_H)$ is rigid in $\R^2$. From $(H-F_H,p_H,q_H)$, we will get back to $G'=G-F$ by rigidity preserving operations. 
To this end first add $v$ to $(H-F_H,p_H,q_H)$ at $p(v)$ by a 2-dimensional 0-loop extension operation with its loops $l_1,l_2$ (with constraints $q(l_1)$ and $q(l_2)$) in $G'=G-F$. This preserves rigidity since $(H-F_H,p_H,q_H)$ is rigid in $\R^2$ and the position of $v$ is fixed after this operation.

The fact that $q_H(s_i)=p(v)-p(x_i)$ implies that for all $s_i\notin F_H$, if we further add the edge $vx_i$, we will get a minimally dependent set of rows indexed by $s_i,vx_i,l_1,l_2$ in the rigidity matrix of the linearly constrained framework we get after this operation.
This holds for each $s_i$ since the constraint corresponding to $s_i$ is determined by the set of constraints $\{vx_i,l_1,l_2\}$; they already restrict the infinitesimal motions of $x_i$ to the line whose normal is $q_H(s_i)=p(x_i)-p(v)$. Thus we can sequentially add $vx_i\notin F$ and remove $s_i$, for each $i$, and preserve rigidity at each step; see Figure \ref{fig:replacing_si_by_vxi} for an illustration when there are two $s_i$ which are not in $F_H$.
\begin{figure}
    \centering
    \begin{tikzpicture}[scale=.8]
        \node[roundnode] (v) at (0,2) [label=right:$v$]{}
            edge[in=70, out=40, loop, label={[label distance=.2cm]55:$l_2$}] ()
            edge[in=110, out=140, loop, label={[label distance=.2cm]125:$l_1$}] ();
        \draw (0,0) ellipse (2cm and 1cm);
        \node at (0,0) [label={[label distance=.6cm]270:$(H-F_H\text{,}p_H\text{,}q_H)$}]{};
        \node[roundnode] (x1) at (-1,0) [label=left:$x_1$]{}
            edge[in=255, out=285, loop, label={[label distance=.2cm]275:$s_1$}] (x1);
        \node[roundnode] (x2) at (1,0) [label=right:$x_2$]{}
            edge[in=255, out=285, loop, label={[label distance=.2cm]265:$s_2$}] (x1);
        \begin{scope}[xshift=5cm]
            \node[roundnode] (v) at (0,2) [label=right:$v$]{}
            edge[in=70, out=40, loop, label={[label distance=.2cm]55:$l_2$}] ()
            edge[in=110, out=140, loop, label={[label distance=.2cm]125:$l_1$}] ();
        \draw (0,0) ellipse (2cm and 1cm);
        \node[roundnode] (x1) at (-1,0) [label=left:$x_1$]{}
            edge[] (v);
        \node[roundnode] (x2) at (1,0) [label=right:$x_2$]{}
            edge[in=255, out=285, loop, label={[label distance=.2cm]265:$s_2$}] (x1);            
        \end{scope}
        \begin{scope}[xshift=10cm]
            \node[roundnode] (v) at (0,2) [label=right:$v$]{}
            edge[in=70, out=40, loop, label={[label distance=.2cm]55:$l_2$}] ()
            edge[in=110, out=140, loop, label={[label distance=.2cm]125:$l_1$}] ();
        \draw (0,0) ellipse (2cm and 1cm);
        \node[roundnode] (x1) at (-1,0) [label=left:$x_1$]{}
            edge[] (v);
        \node[roundnode] (x2) at (1,0) [label=right:$x_2$]{}
            edge[] (v);            
        \end{scope}
    \end{tikzpicture}
    \caption{An illustration of Case 1 in the proof of Theorem \ref{thm:weakly6balanced}. For simplicity suppose only $s_1,s_2$ are not in $F_H$. The framework inside the ellipse on the left is $(H -F_H,p_H,q_H)$. We can add $v$ with $l_1$ and $l_2$ and preserve rigidity in $\R^2$. Then when we add $vx_1$, we get a circuit $\{s_1,vx_1,l_1,l_2\}$ (since $q_H(s_i)=p(x_i)-p(v)$) and so we can remove $s_1$ and preserve rigidity in $\R^2$ and this gives the framework in the middle. Repeating for $s_2$ gives the framework on the right, which is $(G-F,p,q)$ (possibly with less loops at $v$) by our construction.}
    \label{fig:replacing_si_by_vxi}
\end{figure}
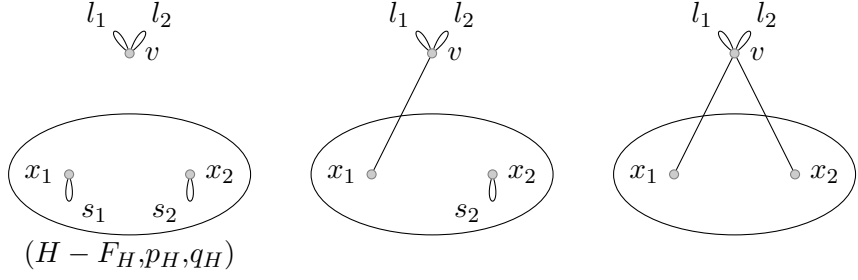
If $v$ has loops other than $l_1,l_2$ we can now add them as well. Since this process ends with $(G-F,p,q)$, we deduce that $(G-F,p,q)$ is rigid in $\R^2$ as required.\\

\noindent\textbf{Case 2.} $v$ has at most one loop in $G'=G-F$.\\

The fact that $l_1$ and $l_2$ are loops at $v$ in $G$ implies that at least one of them, say $l_1$, is in $F$. Suppose $V\setminus (T\cup\{v\})=\emptyset$. Then since $|T|\leq 4$, the fact that $G$ is weakly 6-balanced implies that $G$ contains $K_m^{[7-m]}$ as a spanning subgraph for some $m\leq 5$ and the statement easily follows. Hence there exists a vertex $x\in V\setminus (T\cup\{v\})$. The fact that $N_G(v)\subseteq T$ implies $x\notin N_G(v)$. Consider the graph $H'=G-l_1+vx$.

\begin{claim}\label{cla:H_weakly_6_balanced2}
The graph $H'$ is weakly 6-balanced.    
\end{claim}

\begin{proof}[Proof of Claim \ref{cla:H_weakly_6_balanced2}]
Suppose the contrary and let $U$ be a set for which $H'-U$ has a connected component $C$ with less than $6-|U|$ loops. Since $E(G)\subseteq E(H')$ and $G$ is weakly 6-balanced, we must have $v\in V(C)$. The fact that $v$ has at least one loop in $H'$ implies $|U|\leq 4$. Suppose $|V(C)|\geq 2$ and take a vertex $u$ in $V(C)\setminus \{v\}$. Then the component containing $u$ in $G-(U\cup\{v\})$ has less than $6-|U\cup\{v\}|$ loops, a contradiction. Thus $V(C)=\{v\}$ and we forcefully have $N_{H'}(v)\subseteq U$. Hence the facts that $|N_{H'}(v)|=|N_G(v)|+1$ and $v$ has one more loop in $G$ than it has in $H'$, imply that $C$ is also a connected component of $G-(U\setminus \{x\})$ with less than $6-|U\setminus\{x\}|$ loops, a contradiction to the fact that $G$ is weakly 6-balanced.
\end{proof}

Let us set $F':=\{vx\}\cup(F\setminus\{l_1\})$. Combining the minimality of $|L|$ and Claim \ref{cla:H_weakly_6_balanced2}, the graph $H'-F'$ is $\mathcal{L}_2$-rigid and this completes the proof since $H'-F'=G-F$.
\end{proof}

We also give an alternative proof utilising the discharging method.

\begin{proof}[Proof of Theorem \ref{thm:weakly6balanced} using the discharging method]
Let $G-F$ be a counterexample with $|V|+|L|+|F|$ as small as possible and with respect to this $|E|$ being as large as possible. Let $F_e=F\cap E$, $F_l=F\cap L$ and
$H=G-F=(V,E\setminus F_e,L\setminus F_l)$.
Combining the fact that $H=G-F$ is not $\mathcal{L}_2$-rigid and \cite[Theorem 2.7]{G}, there exists a set $L'\subseteq (L\setminus F_l)$
and an admissible 1-thin cover $\X=\{X_0,X_1,\ldots,X_k\}$ of $H-L'$ such that 
\begin{equation}\label{eq:less_than_2V}
 r_2^{lc}(H)=\val(\X)=|L'|+2|X_0|+\sum_{i=1}^k(2|X_i|-3)<2|V|.   
\end{equation}
We choose $\X$ such that the set $X_0$, i.e., the looped member is maximal. Note that by the minimality of $|V|+|L|+|F|$, no edge or loop in $L'\cup F$ is induced by the vertices in $X_0$.
Let $e_i=x_iy_i$, $1\leq i\leq |F_e|$ denote the edges in $F_e$ and put $X_{k+i}=\{x_i,y_i\}$, for all $1\leq i\leq |F_e|$. Let also $\bar L=L'\cup F_l$. Consider the family
$\mathcal{X}'=\{\bar L\}\cup\{X_i:0\leq i\leq k+|F_e|\}$. In order to apply the discharging method we will assign a starting charge $\mu(X)$ to each set $X$ in this family and reallocate the charges in order to get a contradiction. To this end let $\mu(\bar L)=|\bar L|$, $\mu(X_0)=2|X_0|$ and $\mu(X_i)=2|X_i|-3$ for all $1\leq i\leq k+|F_e|$. Eq.\ (\ref{eq:less_than_2V}) implies that the total charge $\mu_{\text{total}}$ satisfies
\begin{equation}\label{eq:total_mu}
    \mu_{\text{total}}=\mu(\bar L)+\sum_{i=0}^{k+|F_e|}\mu(X_i)=\val(\X)+|F|<2|V|+3.
\end{equation}
Let $\mathcal{X}'(x)$ be the set of members of $\mathcal{X}'$ that contain $x\in V$.
For a vertex $x\in X$ let $\sigma_x(X)$ denote the charge $x$ gets from $X$. The following are the discharging rules we shall use.
\begin{itemize}
    \item R1 - For each $x\in V\setminus X_0$, $\sigma_x(\bar L)=\frac{l(x)}{2}$,
    \item R2 - For $x\in X_0$, if $|\X'(x)|\geq 2$ then $\sigma_x(X_0)=\frac{3}{2}$; and $\sigma_x(X_0)=2$ when $\X'(x)=\{X_0\}$,
    \item R3 - For  $x\in X_i$, $1\leq i\leq k+|F_e|$
        \begin{itemize}
            \item[$\diamond$] R3A - If $|X_i|\leq6$, $\sigma_x(X_i)=2-\frac{3}{|X_i|}$,
            \item[$\diamond$] R3B - If $|X_i|\geq 7$, $\sigma_x(X_i)=\frac{3}{2}$ if $|\X'(x)|\geq 2$, and $\sigma_x(X_i)=2-\frac{l(x)}{2}$ if $\X'(x)=\{X_i\}$.
        \end{itemize}
\end{itemize}
It is straightforward to see that every charge transfer $\sigma_x(X)$ is non-negative except the charge transferred via (R3B). In fact (R3B) transfers a non-negative charge as well. The only negative charge possibility comes from the part $\sigma_x(X_i)=2-\frac{l(x)}{2}$ when $l(x)\geq 5$. We shall show that $l(x)\leq 4$ for every $x\in V\setminus X_0$. For a contradiction suppose this is not the case. Then $x$ would have at least two loops in $G-F$. This then contradicts the choice of $\X$ as we can remove the loops at $x$ (loops in $G-F$) from $L'$, set the looped member of the cover as $X_0\cup\{x\}$ and construct another admissible 1-thin cover with a value less than or equal to the value of $\X$. This contradicts either the fact that $\X$ hits the rank or the maximality of $X_0$. Thus every vertex in $V\setminus X_0$ has at most four loops in $G$ and so every charge transfer is non-negative. Let $\mu'(x)$ denote the final charge of $x$ after the discharging process. Let $Y\subset X_0$ denote the set of vertices $y\in X_0$ with $N(y)\subseteq X_0$ and put $Z=X_0\setminus Y$.

\begin{claim}\label{cla:final_charges}
We have $\mu'(\bar L)=\frac{|\bar L|}{2}$, $\mu'(X_0)=\frac{|Z|}{2}$, $\mu'(X_i)\geq 0$ for all $1\leq i\leq k+|F_e|$ and $\mu'(v)\geq 2$ for all $v\in V$.
\end{claim}

\begin{proof}[Proof of Claim \ref{cla:final_charges}]
The statements that $\mu'(\bar L)=\frac{|\bar L|}{2}$, that $\mu'(X_0)=\frac{|Z|}{2}$ and that $\mu'(X_i)\geq 0$ for all $1\leq i\leq k+|F_e|$ with $|X_i|\leq 6$ directly follow from (R1), (R2) and (R3A), respectively.
Now let us prove $\mu'(X_i)\geq 0$ for all $|X_i|\geq 7$. Let $Z_i\subseteq X_i$ denote the set of vertices $x$ in $X_i$ for which 
$|\X'(x)|\geq 2$ and $Y_i=X_i\setminus Z_i$.
Let $l(Y_i)$ denote the number of loops induced by the vertices in $Y_i$. Since $G$ is weakly 6-balanced we have $|Z_i|+l(Y_i)\geq 6$. Using this and applying (R3B), we get $$\mu'(X_i)=2|X_i|-3-(\frac{3|Z_i|}{2}+2|Y_i|-\frac{l(Y_i)}{2})\geq \frac{|Z_i|+l(Y_i)}{2}-3\geq 0.$$

Finally let us prove $\mu'(v)\geq 2$ for all $v\in V$. If $v\in X_0$ then $\mu'(v)\geq 2$ easily follows. Therefore we may assume $v\notin X_0$. By relabeling, if necessary, we may assume $X_1,X_2,\ldots,X_m$ are the sets containing $v$ such that $|X_1|\geq |X_2|\geq \cdots \geq |X_m|$. If $|X_1|\geq 7$ and $m=1$, then $\mu'(v)=2-\frac{l(v)}{2}+\frac{l(v)}{2}= 2$ after $v$ gets $2-\frac{l(v)}{2}$ from $X_1$ and $\frac{l(v)}{2}$ from $\bar L$. If  $|X_1|\geq 7$ and $m\geq 2$, then $\mu'(v)\geq  \frac{3}{2}+\frac{1}{2}= 2$ after $v$ gets $\frac{3}{2}$ from $X_1$ and at least $\frac{1}{2}$ from $X_2$.

Thus we may assume $|X_1|\leq 6$. First suppose $|X_1|=6$, then by weak 6-balancedness either $l(v)\geq 1$ or $m\geq 2$ holds. Both cases imply $\mu'(v)\geq \frac{3}{2}+\frac{1}{2}=2$ after $v$ gets $\frac{3}{2}$ from $X_1$ and $\frac{1}{2}$ from either $X_2$ or $\bar L$. Next suppose $3\leq |X_1|\leq 5$. Then either $m\geq 2$ with $|X_2|\geq 3$ or $l(v)+(m-1)\geq 2$ holds by weak 6-balancedness. For the former case $\mu'(v)\geq 1+1=2$ after $v$ gets at least $1$ from $X_1$ and at least $1$ from $X_2$. For the latter case
$\mu'(v)\geq 1+\frac{l(v)+(m-1)}{2}\geq 2$ after $v$ gets at least $1$ from $X_1$, $\frac{l(v)}{2}$ from $\bar L$ and at least $\frac{1}{2}$ from each $X_i$, $2\leq i\leq m$. Finally suppose $|X_1|=2$. This forces $l(v)+m\geq 6$ by weak 6-balancedness. Thus $\mu'(v)= \frac{l(v)}{2}+\frac{m}{2}\geq 3$ after $v$ gets $\frac{l(v)}{2}$ from $\bar L$ and $\frac{1}{2}$ from each $X_i$, $1\leq i\leq m$. This completes the proof of the claim.
\end{proof}

Suppose $\X=\{X_0\}$. Then $V=X_0$ and by the maximality of $|E|$, the graph $(V,E)$ is a complete graph. Then for the case $|V|\leq 5$, weak 6-balancedness forces $G$ to contain $K_{|V|}^{[7-|V|]}$ as a spanning subgraph for which we trivially have $G-F$ is $\mathcal{L}_2$-rigid. For the case $|V|\geq 6$, there are at least six vertices with loops in $G$ and since the underlying simple graph is complete this again trivially gives that $G-F$ is $\mathcal{L}_2$-rigid. Thus we may assume $\X$ has members other than $X_0$. Since the loops contained in $\bar L$ are incident with vertices in $V\setminus X_0$, and $Z\subseteq X_0$ is the set of vertices in $X_0$ having at least one neighbour in $V\setminus X_0$, the weak 6-balancedness of $G$ gives $|Z|+|\bar L|\geq 6$. 

Now Claim \ref{cla:final_charges} and Eq.\ (\ref{eq:total_mu}) imply that
\begin{align*}
    2|V|+3>&\mu_{\text{total}}\\
          =&\mu'(\bar L)+\mu'(X_0)+\sum_{i=1}^{k+|F_e|}\mu'(X_i)+\sum_{v\in V}\mu'(v)\\
          \geq&\frac{\bar L}{2}+\frac{|Z|}{2}+0+2|V|\\
          =&2|V|+3,
\end{align*}
a contradiction where the last inequality follows from the fact that $|Z|+|\bar L|\geq 6$.
\end{proof}

We can now combine Theorems \ref{thm:glob_char} and \ref{thm:weakly6balanced} to obtain the following.

\begin{thm}\label{thm:weakly6balancedglobal}
    Every weakly 6-balanced looped simple graph is globally $\mathcal{L}_2$-rigid.
\end{thm}

Recall that a graph $G$ is called {\em essentially $k$-balanced} if $|V|\geq k+1$ and for every $T\subset V$ with $|T|<k$,
each component of $G-T$ with at least one simple edge has at least one loop. Let $G=(V,E,L)$ be a looped simple graph and suppose that for each simple edge $xy\in E$ either $x$ or $y$ has a loop, that is each simple edge is incident with a loop. Then $G$ is trivially essentially $(|V|-1)$-balanced. Now consider Figure \ref{fig:essential_balanced}. The graph $G$ is 2-balanced and fails being 3-balanced as we can remove $v_1,v_2$ and leave the component consisting of $v_0$ without loops. We can also say that $G$ is essentially 4-balanced, because each simple edge is incident with a loop. The graph $H$ is only 1-balanced and, by the same reason as $G$, it is essentially 4-balanced. The graph $K$ can be obtained from $H$ by adding the edge $v_0v_1$. For the graph $K$, removal of $v_2,v_4$ leaves the edge $v_0v_1$ as a loopless component. Thus, $K$ is not essentially 3-balanced. Since $H$ is essentially 4-balanced, this tells us that edge addition may decrease essential balancedness.

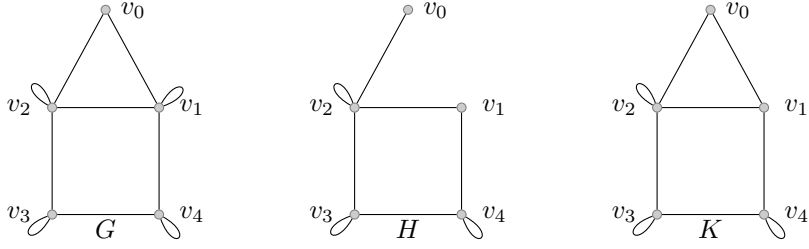
\begin{figure}[hbt]
\begin{center}

\begin{tikzpicture}[scale=1,font=\small]
%\fill[pattern=north west lines,opacity=.6,draw] (0,0) circle [radius=4];
\begin{scope}[rotate=45]

\foreach \i in {1,2,3,4}
{
	\node[roundnode] (\i) at ({(\i-1)*90}: 1cm)  [] {};
}

\foreach \i in {1,4}
{
	\node[label={[label distance=0cm]0:$v_\i$}] at ({(\i-1)*90}: 1cm) {};
}
\foreach \i in {2,3}
{
	\node[label={[label distance=0cm]180:$v_\i$}] at ({(\i-1)*90}: 1cm) {};
}
\foreach \i/\j in {1/2,2/3,3/4,4/1}
{
	\draw[-] (\i) to (\j);
}
\foreach \i in {1,2,3,4}
{
    \pgfmathsetmacro\startp{(\i-1)*90-20}
	\pgfmathsetmacro\endp{(\i-1)*90+20}
	\draw[-, in=\startp, out=\endp,loop] (\i) to (\i);
}
\end{scope}
\node[roundnode,label={[label distance=0cm]0:$v_0$}] (0) at (90: 2cm)  [] {}
    edge[] (1)
    edge[] (2);
\node [label={[label distance=0.5cm]270:$G$}] {};

\begin{scope}[xshift=4cm]
\begin{scope}[rotate=45]

\foreach \i in {1,2,3,4}
{
	\node[roundnode] (\i) at ({(\i-1)*90}: 1cm)  [] {};
}

\foreach \i in {1,4}
{
	\node[label={[label distance=0cm]0:$v_\i$}] at ({(\i-1)*90}: 1cm) {};
}
\foreach \i in {2,3}
{
	\node[label={[label distance=0cm]180:$v_\i$}] at ({(\i-1)*90}: 1cm) {};
}
\foreach \i/\j in {1/2,2/3,3/4,4/1}
{
	\draw[-] (\i) to (\j);
}
\foreach \i in {2,3,4}
{
    \pgfmathsetmacro\startp{(\i-1)*90-20}
	\pgfmathsetmacro\endp{(\i-1)*90+20}
	\draw[-, in=\startp, out=\endp,loop] (\i) to (\i);
}
\end{scope}
\node[roundnode,label={[label distance=0cm]0:$v_0$}] (0) at (90: 2cm)  [] {}
    edge[] (2);
\node [label={[label distance=0.5cm]270:$H$}] {};
\end{scope}

\begin{scope}[xshift=8cm]
\begin{scope}[rotate=45]

\foreach \i in {1,2,3,4}
{
	\node[roundnode] (\i) at ({(\i-1)*90}: 1cm)  [] {};
}

\foreach \i in {1,4}
{
	\node[label={[label distance=0cm]0:$v_\i$}] at ({(\i-1)*90}: 1cm) {};
}
\foreach \i in {2,3}
{
	\node[label={[label distance=0cm]180:$v_\i$}] at ({(\i-1)*90}: 1cm) {};
}
\foreach \i/\j in {1/2,2/3,3/4,4/1}
{
	\draw[-] (\i) to (\j);
}
\foreach \i in {2,3,4}
{
    \pgfmathsetmacro\startp{(\i-1)*90-20}
	\pgfmathsetmacro\endp{(\i-1)*90+20}
	\draw[-, in=\startp, out=\endp,loop] (\i) to (\i);
}
\end{scope}
\node[roundnode,label={[label distance=0cm]0:$v_0$}] (0) at (90: 2cm)  [] {}
    edge[] (1)
    edge[] (2);
\node [label={[label distance=0.5cm]270:$K$}] {};
\end{scope}
%\node[roundnode] at (-0.7cm, 0cm) (u) [] {};
%\node[roundnode] at (0.7cm, 0cm) (v) [] {};
%\node[roundnode] at (0cm, 0.7cm) (a) [] {}
%	edge[] (u)
%	edge[] (v);
\end{tikzpicture}
\end{center}
\caption{Essential balancedness.}
\label{fig:essential_balanced}
\end{figure}

Theorem \ref{thm:main} relaxes the ``weakly 6-balancedness'' condition in Theorem \ref{thm:weakly6balancedglobal} by inserting essential balancedness. We are now ready to give the proof.

\begin{proof}[Proof of Theorem \ref{thm:main}]
First note that weakly 4-balanced graphs satisfy condition (ii) in the statement of Theorem \ref{thm:glob_char}. Thus, by that same theorem, showing $G$ is redundantly $\mathcal{L}_2$-rigid is enough to guarantee global $\mathcal{L}_2$-rigidity. For a contradiction let $G$ be a looped simple graph which is not redundantly $\mathcal{L}_2$-rigid with $|V|+|L|$ being as small as possible.

\begin{claim}\label{cla:4balanced}
For each $v\in V$, the inequality $l(v)\leq 1$ holds and so the graph $G$ is 4-balanced.
\end{claim}

\begin{proof}[Proof of Claim \ref{cla:4balanced}]
Suppose the contrary. Pick a vertex $v$ with at least two loops, say $l_1,l_2$. Consider the graph $G-l_1$. The fact that $v$ still has a loop, namely $l_2$, in $G-l_1$ implies that $G-l_1$ is essentially 6-balanced. Thus by the minimality of $|V|+|L|$, $G-l_1$ cannot be weakly 4-balanced. Therefore there exists a set $T$ such that some component $C$ in $G-l_1-T$ has less than $4-|T|$ loops. Since $G$ is weakly 4-balanced, $v\in V(C)$ and so $|T|\leq 2$. If there exists a vertex $u$ in $C$ other than $v$, then the connected component containing $u$ of $G-(T\cup\{v\})$ would have less than $4-|T|$ loops, a contradiction. Hence $V(C)=\{v\}$, and so $N(v)\subseteq T$ and $|N(v)|\leq 2$.

Consider the graph $H$ obtained from $G$ by deleting $v$ and adding a loop at each $x\in N_G(v)$. We will show that $H$ is weakly 4-balanced and essentially 6-balanced.

Consider an arbitrary connected component $C_H$ of $H-S$ for some $|S|\leq 4$. We will show that $C_H$ has at least $4-|S|$ loops. If $V(C_H)\cap N_G(v)=\emptyset$, then $C_H$ is still a connected component in $G-S$ and so the weak 4-balancedness of $G$ gives that $C_H$ has at least $4-|S|$ loops. If there exists $x\in V(C_H)\cap N_G(v)$, the fact that $x$ has at least one loop in $H$ implies that $C_H$ has at least $4-|S|$ loops or $|S|\leq 2$, so we may assume $|S|\leq 2$. Since $G$ is weakly 4-balanced the connected component containing $x$ in $G-(S\cup\{v\})$ has at least $4-|S\cup\{v\}|$ loops. This implies that $C_H$ contains at least $4-|S|$ loops as $x$ has one more loop in $H$ than it has in $G$. Therefore $H$ is weakly 4-balanced.

Now consider an arbitrary connected component $C_H$ of $H-T$ for some $|T|<6$. If $V(C_H)\cap N_G(v)\neq \emptyset$, then $C_H$ has at least one loop (every vertex in $N_G(v)$ has a loop in $H$). If $V(C_H)\cap N_G(v)= \emptyset$ then $C_H$ is also a connected component in $G-T$ and the fact that $G-T$ is essentially 6-balanced implies $C_H$ has at least one loop. Thus $H$ is essentially 6-balanced. The minimality of $|V|+|L|$ now implies that $H$ is redundantly $\mathcal{L}_2$-rigid.

Now, for all $f\in E\cup L$, we shall show $G-f$ is $\mathcal{L}_2$-rigid, hence $G$ is redundantly $\mathcal{L}_2$-rigid. First suppose that $f\in E(H)\cup L(H)$. Then $H-f$ is $\mathcal{L}_2$-rigid since $H$ is redundantly $\mathcal{L}_2$-rigid. We can now add $v$ back by a 2-dimensional $k$-loop extension for some $k=1$ or $k=2$ depending on whether $f$ is a loop a vertex in $N_G(v)$ or not. This shows that $G-f$ is $\mathcal{L}_2$-rigid by Lemma \ref{lem:ddim-kloop}.
Next suppose $f$ is incident with $v$. 
Let $l_x$ be the loop at $x$ we added when constructing $H$ from $G$ if $f=vx$, and for some arbitrary $x\in N_G(v)$ if $f$ is a loop at $v$. Then, since $H-l_x$ is $\mathcal{L}_2$-rigid, we can add $v$ to $H-l_x$ by a 2-dimensional 1-loop extension and preserve $\mathcal{L}_2$-rigidity by Lemma \ref{lem:ddim-kloop}.
\end{proof}

Since $G$ is not redundantly $\mathcal{L}_2$-rigid there exists an $f\in E\cup L$ such that at least one endpoint of $f$ is not contained in a $\mathcal{L}_2$-rigid subgraph of $G-f$. Then by \cite[Theorem 2.7]{G}, there exists a set $L'\subset L-f$ and an admissible 1-thin cover
$\mathcal{X}=\{X_0,X_1,\ldots,X_k\}$ of $G-L'-f$ such that
\begin{equation}\label{eq:leq2V}
r_2^{lc}(G-f)=\val(\mathcal{X})=|L'|+2|X_0|+\sum_{i=1}^k(2|X_i|-3)<2|V|.
\end{equation}
We choose $\mathcal{X}$ in a way that $X_0$, i.e. the looped member, is maximal.
Since $\X$ is a 1-thin cover from which we obtain the rank of $G-f$, we deduce that either $X_0=\emptyset$ or $G[X_0]$ is $\mathcal{L}_2$-rigid. This shows us that $f\notin E(X_0)\cup L(X_0)$. Moreover, the choice of $\X$ implies $|E(X_0)\cup L(X_0)|\geq 2|X_0|$ and $|E(X_i)|\geq 2|X_i|-3$ for all $1\leq i\leq t$ as otherwise $\X$ would not be the minimising cover, i.e., the cover whose value gives the rank.

If $f=uv\in E$, let $\X'=\X\cup\{\{u,v\}\}$ and $\bar L= L'$, and if $f\in L$, let $\X'=\X$ and $\bar L=L'\cup\{f\}$. That is, add the endpoint set to $\X$ when $f$ is a simple edge and add $f$ to $L'$ if it is a loop. Using the choice of $L'$ and the fact that $f\notin E(X_0)\cup L(X_0)$ we deduce that every loop in $\bar L$ is induced by some vertex in $V\setminus X_0$.

Let $\mathcal{X}'=\{X_0,X_1,\ldots,X_t\}$ by setting $t=k$ when $\mathcal{X}'=\mathcal{X}$ and otherwise by setting $t=k+1$ and $X_{k+1}=\{u,v\}$.  We shall assign some initial charges and some set of reallocation rules. Let $\mu(x)$ denote the initial charge for $x$.
Set $\mu(\bar L)=|\bar L|$, $\mu(X_0)=2|X_0|$ and $\mu(X_i)=2|X_i|-3$ for $1\leq i\leq t$. Let $\mu_{\text{total}}$ denote the sum of the charges. Then we have
\begin{equation}\label{eq:main_eq}
\mu_{\text{total}}=\mu(\bar L)+\sum_{i=0}^t\mu(X_i)=|\bar L|+2|X_0|+\sum_{i=1}^t\mu(X_i)=\val(\mathcal{X})+1<2|V|+1.
\end{equation}

Let $\mathcal{X}'(x)$ be the set of members of $\mathcal{X}'$ that contain $x\in V$. For $X\in \mathcal{X}'$ and $x\in X$,
let $\sigma_x(X)$ denote the charge $x$ gets from $X$ and $\sigma_x(\bar L)$ denote the charge $x$ gets from $\bar L$. We use the following discharging rules:
\begin{itemize}
\item R1 - If $x$ has a loop contained in $\bar L$, then $\sigma_x(\bar L)=\frac{1}{2}$.
\item R2 - For $x\in X_0$, if $x\in X_i$ for some $1\leq i\leq t$ then $\sigma_x(X_0)=\frac{3}{2}$; else $\sigma_x(X_0)=2$.
\item R3 - If $x\in X\neq X_0$ with $|X|=2$, $\sigma_x(X)=\frac{1}{2}$.
\item R4 - If $x\in X\neq X_0$ with $|X|=3$, $\sigma_x(X)=1$.
\item R5 - For $X\neq X_0$ with $|X| = 4$:
	\begin{itemize}
		\item[$\diamond$] R5A - If $X$ contains a vertex $x$ such that $N[x]\subset X$ or a vertex $x$ such that $\mathcal{X}'(x)=\{X,Y\}$ with $|Y|=2$
		then for all $x$ with $N[x]\subset X$, $\sigma_x(X)=\frac{3}{2}$; for all $x$ with $\mathcal{X}'(x)=\{X,Y\}$ with $|Y|=2$, $\sigma_x(X)=\frac{3}{2}-\frac{l(x)}{2}$
		and other vertices in $X$ get 1 from $X$. 
		\item[$\diamond$] R5B - Otherwise, $\sigma_x(X)=\frac{5}{4}$ for all $x\in X$.
	\end{itemize}
\item R6 - For a set $X\neq X_0$ with $|X|\geq 5$ and vertex $x\in X$:
\begin{itemize}
	\item[$\diamond$] R6A - If $N[x]\subset X$ then $\sigma_x(X)=2-\frac{l(x)}{2}$,
	\item[$\diamond$] R6B - if $\mathcal{X}'(x)=\{X,Y\}$ and $|Y|=2$ then $\sigma_x(X)=\frac{3}{2}-\frac{l(x)}{2}$,
	\item[$\diamond$] R6C - otherwise, $\sigma_x(X)=1$.
\end{itemize}
\end{itemize}
By the choice of $\X$ each simple graph $(X_i,E(X_i))$, $1\leq i\leq t$ contains a spanning 2-rigid subgraph as otherwise we could decompose $X_i$ into smaller sets and obtain an admissible 1-thin cover with smaller value. Thus each $x\in X_i$, $1\leq i\leq t$, has a neighbour in $X_i$.
Similarly,  the looped simple graph $(X_0,E(X_0),L(X_0))$ contains a spanning $2$-tight subgraph as otherwise we could decompose $X_0$ into smaller sets and obtain an admissible 1-thin cover with less value. Hence the fact that $l(x)\leq 1$ by Claim \ref{cla:4balanced} implies that $x\in X_0$ must have a neighbour in $X_0$.
These further imply that $X_0=\emptyset$ or $|X_0|\geq 3$. Therefore, the set $Y$ in the rules (R5-R6) cannot be $X_0$. 
For a vertex $x\in V$, let $\mu'(x)=\sigma_x(\bar L)+\sum_{i=0}^t\sigma_x(X_i)$ i.e., $\mu'(x)$ is the total charge $x$ gets from $\bar L$
and $X_i$ for $0\leq i\leq t$. Let $\mu'(\bar L)$ and $\mu'(X_i)$ denote the final charge of $\bar L$ and $X_i$ for $0\leq i\leq t$.

\begin{claim}\label{cla:max_loops}
$l(X_i)\leq 3$ for all $1\leq i\leq t$ and if the equality holds the $f$ is a loop at a vertex in $X_i$.
\end{claim}

\begin{proof}[Proof of Claim \ref{cla:max_loops}]
This follows from the maximality of $X_0$.
\end{proof}

\begin{claim}\label{cla:charge_vertex}
$\mu'(x)\geq 2$ for all $x\in V$.
\end{claim}

\begin{proof}[Proof of Claim \ref{cla:charge_vertex}]
If $x\in X_0$ then $\mu'(x)\geq 2$ after $x$ gets 2 from $X_0$ by (R2) or $\frac{3}{2}$ from $X_0$ and at least $\frac{1}{2}$ from $X_i$ for some $i$ by (R3-R6). Hence we may assume $x\notin X_0$. By Claim \ref{cla:4balanced} $|N(x)|\geq 3$ and if $|N(x)|=3$ then we have $l(x)=1$.
\\\noindent\textbf{Case 1.} $|\mathcal{X}'(x)|\geq 4$.\\
Then $\mu'(x)\geq 4\cdot\frac{1}{2}=2$ after $x$ gets at least $\frac{1}{2}$ from each set in $\X'(x)$ by (R2-R6).

\noindent\textbf{Case 2.} $|\mathcal{X}'(x)|=3$.\\
First suppose $\X'(x)$ contains a member $X$ of size at least three. Then $\mu'(x)\geq 1+2\cdot \frac{1}{2}=2$ after $x$ gets at least $1$ from $X$ by (R4-R6) and at least $\frac{1}{2}$ for each set in $\X'(x)\setminus\{X\}$ by (R2-R6).

Now suppose all members in $\X'(x)$ have size two. This forces $|N(x)|=3$ and so $l(x)=1$ by 4-balancedness. Then $\mu'(x)= 3\cdot \frac{1}{2}+\frac{1}{2}=2$ after $x$ gets $\frac{1}{2}$ from each member in $X'(x)$ and $\frac{1}{2}$ from $\bar L$.

\noindent\textbf{Case 3.} $|\mathcal{X}'(x)|=2$.\\
Then $x$ is contained in a set $X$ of size at least three since $|N(x)|\geq 3$. Suppose $|X|=3$.
Then when $|N(x)|=3$ we have $\mu'(x)= 1+\frac{1}{2}+\frac{1}{2}=2$ after $x$ gets 1 from $X$ by (R4) and $\frac{1}{2}$ from the other member of $\X'(x)$ and $\frac{1}{2}$ from $\bar L$ as $|N(x)|=3$ forces $l(x)=1$. For the case $|N(x)|\geq 4$ we have $\mu'(x)\geq 1+1$ after $x$ gets 1 from $X$ by (R4) and at least 1 from the other member in $\X'(x)$ by (R4-R6). 

Now suppose $|X|\geq 4$. For the case $\X'(x)=\{X,Y\}$ with $|Y|=2$ we have $\mu'(x)\geq (\frac{3}{2}-\frac{l(x)}{2})+\frac{1}{2}+\frac{l(x)}{2}=2$ after $x$ gets $\frac{3}{2}-\frac{l(x)}{2}$ by (R5A or R6B), $\frac{1}{2}$ from $Y$ by (R3) and $\frac{l(x)}{2}$ from $\bar L$ by (R1). Otherwise, (R5B or R6C) applies to $X$ and the other member in $\X'(x)$ would have size at least three. This gives $\mu'(x)\geq 1+1=2$ after $x$ gets at least 1 from $X$ by (R5B or R6C) and at least 1 from the other member of $\X'(x)$ by (R4-R6).  

\noindent\textbf{Case 4.} $|\mathcal{X}'(x)|=1$.\\
Then the unique member $X$ in $\mathcal{X}'(x)$ has size at least four as $|N(x)|\geq 3$. 
If $|X|=4$ then $\mu'(x)=\frac{3}{2}+\frac{1}{2}$ after $x$ gets $\frac{3}{2}$ from $X$ by (R5A) and $\frac{1}{2}$ from $\bar L$ as $|X|=4$ forces $|N(x)|=3$ and so $l(x)=1$. If $|X|\geq 5$ then $\mu'(x)=(2-\frac{l(x)}{2})+\frac{l(x)}{2}=2$ after $x$ gets $2-\frac{l(x)}{2}$ from $X$ by (R6A) and $\frac{l(x)}{2}$ from $\bar L$.
\end{proof}

\begin{claim}
For a member $X_{i}\in\mathcal{X}'$ such that $1\leq i\leq t$ and $|X_{i}|\leq 3$ we have $\mu'(X_{i})\geq 0$.
\end{claim}
\begin{proof}[Proof of Claim]
If $|X_i|=2$, $\mu'(X_i)=(2\cdot 2-3)-2\cdot \frac{1}{2}=0$. If $|X_i|=3$, $\mu'(X_i)=(2\cdot 3-3)-3\cdot 1=0$. 
\end{proof}

\begin{claim}\label{cla:size4}
For a member $X\in\X'\setminus \{X_0\}$ such that $|X|= 4$. Then we have
\begin{itemize}
\item[(i)]  $\mu'(X)\geq 0$, or
\item[(ii)] $\mu'(X)=-\frac{1}{2}$ and there are at least two loops for which $X$ is the only member of $\X'$ that induces these loops, or
\item[(iii)] $\mu'(X)=-1$ and there are three loops for which $X$ is the only member of $\X'$ that induces these loops.
\end{itemize}
\end{claim}

\begin{proof}[Proof of Claim \ref{cla:size4}]
If (R5B) applies then (i) holds. Hence we may assume (R5A) applies.
Let $T\subseteq X$ be the set of vertices $x$ such that
$N(x)\subseteq X$. Since $|X|=4$, $G$ being 4-balanced forces every vertex $x\in T$ to have a loop and to be adjacent to every vertex in $X\setminus\{x\}$.
Let $S\subseteq X$ be the set of loopless vertices $x$ such that $\sigma_x(X)=\frac{3}{2}$, i.e., $S$ is the set of loopless vertices that have exactly one
neighbour outside of $X$ and is adjacent to every vertex (other than itself) in $X$. Let $P=X\setminus(T\cup S)$. Note that $\{T,S,P\}$ is a partition of $X$.
Since $\mu(X)=5$ we have
\begin{equation}\label{eq:charge4}
\mu'(X)=5-\frac{3}{2}|T|-\frac{3}{2}|S|-|P|
\end{equation}
by (R5A).
If $|S|\geq 2$ then $s_1s_2\in E$ for some $s_1,s_2\in S$ by the definition of $S$ and the fact that $|X|=4$.
Let $S_u:=\{v_s\in V\setminus X: sv_s\in E \text{ for some }s\in S\}$, that is, $S_u$ consists of the unique neighbours $v_s$ of the vertices $s$ in $S$
such that $v_s\notin X$. Note that $|S_u|\leq |S|$. We see that the set
$U=T\cup S_u\cup P$ forms a cut set of size at most four whose removal leaves the component containing the edge $s_1s_2$ without any loops,
contradicting the fact that $G$ is essentially 6-balanced. Hence $|S|\leq 1$.
Suppose $|P|\geq 2$. Then $|T|+|S|\leq 2$ and Eq (\ref{eq:charge4}) gives $\mu'(X)\geq 0$ so (i) holds. Next suppose $|P|=1$. Then $\mu'(X)=-\frac{1}{2}$ and $|T|\geq 2$, so by the definition of $T$ (ii) holds. Finally suppose $|P|=0$. By Claim \ref{cla:max_loops} and the fact that $|S|\leq 1$ we obtain $|T|=3$, $|S|=1$.
Hence $\mu'(X)=-1$, $l(X)=3$ and $X$ is the only member of $\X'$ that induces these loops so (iii) holds. 
\end{proof}

\begin{claim}\label{cla:size_geq5}
Let $X\in \mathcal{X}'\setminus\{X_0\}$ be a set with $|X|\geq 5$. Then we have
\begin{itemize}
\item [(i)] $\mu'(X)\geq 0$, or
\item [(ii)] $\mu'(X)= -\frac{1}{2}$ and there are three loops for which $X$ is the only member of $\X'$ that induces these loops.
\end{itemize}
\end{claim}

\begin{proof}[Proof of Claim \ref{cla:size_geq5}]
Let: $N\subseteq X$ be the set of vertices $x$ such that $\sigma_x(X)=2$, i.e., the set of loopless vertices in $X$ whose neighbourhood are contained
in $X$; $T\subseteq X$ be the set of vertices  $x$ with a loop such that $\sigma_x(X)=\frac{3}{2}$, i.e., the set of looped vertices in $X$ whose neighbourhood are
contained in $X$; $S\subseteq X$ be the set of loopless vertices $x$ such that $\sigma_x(X)=\frac{3}{2}$, i.e., the set of loopless vertices $x$ such that
$\X'(x)=\{X,Y\}$ with $|Y|=2$; and $P\subseteq X$ be the set of vertices $x$ such that $\sigma_x(X)=1$. Note that $\{N,T,S,P\}$ is a partition of $X$ and
\begin{equation}\label{eq:chargeX}
\mu'(X)=\mu(X)-\sum_{x\in X}\sigma_x(X)=2|X|-3-(2|N|+\frac{3}{2}|T|+\frac{3}{2}|S|+|P|).
\end{equation}
Using Eq (\ref{eq:chargeX}) and the fact that $\{N,T,S,P\}$ is a partition of $X$ we may obtain
\begin{align*}
\mu'(X)&=2|X|-3-\big(2|N|+\frac{3}{2}|T|+\frac{3}{2}|S|+|P|\big)\\
	   &=2|X|-3-\big(\frac{3}{2}(|N|+|T|+|S|+|P|)+\frac{|N|-|P|}{2}\big)\\
	   &=2|X|-3-\frac{3}{2}|X|+\frac{|P|-|N|}{2}\\
	   &\geq\frac{|X|}{2}-\frac{|N|-|P|+6}{2}.
\end{align*}
If $|X|\geq |N|-|P|+6$ then we have $\mu'(X)\geq 0$; and if $|P|>|N|$ then $\mu'(X)\geq 0$ as $|X|\geq 5$, so (i) holds. Hence we may assume $|X|\leq |N|-|P|+5$ and $|N|\geq |P|$.
Using the fact that $\{N,T,S,P\}$ is a partition of $X$ we obtain
$$|N|-|P|+5\geq |X| = |T|+|S|+|N|+|P|+|P|-|P|=|T|+|S|+2|P|+(|N|-|P|)$$
and this gives
\begin{equation}\label{eq:ts2pl5}
|T|+|S|+2|P|\leq 5.
\end{equation}
Let $U=T\cup S_u\cup P$ where $S_u:=\{v_s\in V\setminus X: sv_s\in E \text{ for some }s\in S\}$, that is $S_u$ consists of the unique neighbours $v_s$ of the vertices $s$ in $S$
such that $v_s\notin X$. Note that $|S_u|\leq |S|$.\\

\noindent\textbf{Case 1.} $|T|\leq 2$.\\

First suppose $|P|\geq 1$. Using the fact that $|N|\geq |P|$ this gives $|N|\geq 1$. We can also use Eq (\ref{eq:ts2pl5}) and $|P|\geq 1$ to deduce that $|T|+|S|+|P|\leq 4$.
Using $|T|\leq 2$ and Eq (\ref{eq:ts2pl5}) we obtain $|T\cup P|\leq 3$. Pick a vertex $x\in N$. Recall that $N(x)\subseteq X$;
and $|N(x)|\geq 4$ as $G$ is 4-balanced. Using these we deduce that $N(x)\setminus(T\cup P)\neq \emptyset$. Pick $y\in N(x)\setminus(T\cup P)$.
Then $U$ is a cut set of size at most four whose removal leaves the component containing the edge $xy$ without any loops, contradicting the fact that $G$ is essentially 6-balanced.

Next suppose $P=\emptyset$. Since $|T|\leq 2$, we have $|S|\geq 2$ by 4-balancedness as the vertices in $T\cup N$ have all neighbours in $X$ and $N$ consists of loopless vertices.
Pick $s\in S$. Since $G$ is 4-balanced and $s$ is loopless $|N(s)|\geq 4$ and therefore $|N(s)\cap X|\geq 3$.  This tells us that there exists a neighbour of $s$ in $X\setminus T$.
Let $y$ denote such a neighbour. Since $P=\emptyset$, Eq (\ref{eq:ts2pl5}) gives $|T|+|S|\leq 5$ and so $|U|\leq 5$. Then $U$ is a cut set of size at most five whose removal
leaves the component containing the edge $sy$ without any loops, contadicting the fact that $G$ is essentially 6-balanced.\\

\noindent\textbf{Case 2.} $|T|=3$.\\

Claim \ref{cla:max_loops} and the definition of $T$ imply that $l(X)=3$ and $X$ is the only member of $\X'$ that induces these loops.
By Eq (\ref{eq:ts2pl5}) we obtain $|P|\leq 1$. First suppose $|P|=1$. Then Eq (\ref{eq:ts2pl5}) gives $S=\emptyset$ and so $|N|=|X\setminus (T\cup P)|=|X|-4$.
Then we have
\begin{align*}
\mu'(X)&=2|X|-3-\big(2|N|+\frac{3}{2}|T|+|P|\big)\\
	   &=2|X|-3-\big(2(|X|-4)+\frac{3}{2}\cdot 3+1\cdot 1\\
	   &=-\frac{1}{2}
\end{align*}
and (ii) holds.

Next suppose $P=\emptyset$. This gives $|T|+|S|\leq 5$ by Eq (\ref{eq:ts2pl5}) and so $|U|\leq 5$.
For a contradiction suppose also that $N\neq \emptyset$. Pick $x\in N$. Since $N(x)\subset X$ and $|T|=3$,
there exists $y\in N(x)\setminus T$. Then the set $U$ forms a cut set of size at most five whose removal leaves the component containing the edge $xy$ without any loops,
contadicting the fact that $G$ is essentially 6-balanced. Thus $N=\emptyset$ holds. Hence $X=T\cup S$ and this gives $5\leq |X|=|T|+|S|\leq 5$. That is $|S|=2$ and $|X|=5$.
In this case $\mu(X)=7$ and each vertex in $X$ gets $\frac{3}{2}$ from $X$ by the definitions of $T$ and $S$ which gives $\mu'(X)=7-5\cdot \frac{3}{2}=-\frac{1}{2}$, so (ii) holds.
\end{proof}

Let $l_i$ denote the number of loops in $\bar L$ induced by only $X_i\in \X'$, $1\leq i\leq t$ and let $l'$ denote the number of loops in $\bar L$ that are induced by
at least two members $X_i\in \X'$, $1\leq i\leq t$. Note that we have $|\bar L|=l' +\sum_{i=1}^tl_i$.
Let $Z\subset X_0$ be the set of vertices having a neighbour in $V\setminus X_0$, i.e., $Z$ is the set of vertices $z\in X_0$ such that $\sigma_z(X_0)=\frac{3}{2}$.
If $V=X_0$ then this would give $r_2^{lc}(G-f)=2|X_0|=2|V|$, contradicting Eq (\ref{eq:leq2V}). Thus $V\setminus X_0$ is nonempty and now we can combine this with the fact that
$G$ is 4-balanced in order to obtain
\begin{equation}\label{eq:ZplusL}
|Z|+|\bar L|\geq 4.
\end{equation} 

\begin{claim}\label{cla:charge_geq1}
$\frac{|Z|}{2}+\frac{|\bar L|}{2}+\sum_{i=1}^t\mu'(X_i)=\frac{|Z|}{2}+\frac{l'}{2}+\sum_{i=1}^t(\mu'(X_i)+\frac{l_i}{2})\geq 1$.
\end{claim}

\begin{proof}[Proof of Claim \ref{cla:charge_geq1}]
First suppose there are at least two $X_i$ for which $\mu'(X_i)<0$. Then the statement holds as $\mu'(X_i)+\frac{l_i}{2}\geq \frac{1}{2}$ for each such member
by Claims \ref{cla:size4} and \ref{cla:size_geq5}. Next suppose exactly one member $X_i$, $1\leq i\leq t$ has a negative final charge, as above
we have $\mu'(X_i)+\frac{l_i}{2}\geq \frac{1}{2}$ for this fixed $i$. Since $l_i\leq 3$ by Claims \ref{cla:size4} and \ref{cla:size_geq5} using the fact that $G$ is 4-balanced
either $|Z|\geq 1$ or $l'\geq 1$ or $l_j\geq 1$ for some $j\neq i$ thus the statement holds. Finally suppose no member $X_i$, $1\leq i\leq t$ has a negative final charge.
Then the statement holds by Eq (\ref{eq:ZplusL}).
\end{proof}
We can now use Eq (\ref{eq:main_eq}), Claims \ref{cla:charge_vertex} and \ref{cla:charge_geq1}, and the facts that $\mu'(\bar L)=\frac{|\bar L|}{2}$ (by R1)
and that $\mu'(X_0)=\frac{|Z|}{2}$ (by R2) to obtain
\begin{align*}
2|V|+1>\mu_{\text{total}}&=\mu(\bar L)+\sum_{i=0}^t\mu(X_i)\\
&=\mu'(\bar L)+\mu'(X_0)+\sum_{i=1}^t\mu'(X_i) +\sum_{x\in V}\mu'(x)\\
									   &\geq \frac{|\bar L|}{2}+\frac{|Z|}{2}+ \sum_{i=1}^t\mu'(X_i) + 2|V|\\
									   &\geq 2|V|+1,
\end{align*}
a contradiction. This completes the proof.
\end{proof}

\section{Remarks}
We conclude with some examples and remarks that illustrate when our results are tight.

We cannot decrease the number 6 in Theorem \ref{thm:main}, i.e., the essential 6-balancedness. To see this, consider the graph $G$ in \cite[Fig 4]{G}. $G$ is (weakly) 5-balanced (and so weakly 4-balanced) and essentially 5-balanced. However, as explained in detail in \cite{G}, $G$ is not $\mathcal{L}_2$-rigid.

In \cite{GMRWY}, it was shown that graphs satisfying 4-connectivity and essential 6-connectivity
are globally 2-rigid. Theorem \ref{thm:main} can be regarded as its linearly constrained analogue where (weak) 4-balancedness
plays the role of 4-connectivity and essential 6-balancedness plays the role of essential 6-connectivity. In \cite{GMRWY}, it was also shown that
3-connectivity plus essential 9-connectivity is sufficient for global 2-rigidity. It is natural to ask whether the linearly
constrained anologue of this result is true, that is ``is it true that (weak) 3-balancedness plus essential 9-balancedness is sufficient for a graph to be globally
$\mathcal{L}_2$-rigid?". The examples below show that there are an infinite number of graphs
proving this is false. They also show that strengthening to essential $k$-balancedness for some $k>9$ is not
enough to make the corresponding graphs satisfy global $\mathcal{L}_2$-rigidity.\\

\noindent\textbf{Example 1.} Let $G_k=(V,E,L)$ be the looped simple graph obtained from a $k$-cycle for some $k\geq 3$ by adding a loop at each vertex, see Figure \ref{fig:G10}.
Since the $k$-cycle is 2-connected and each vertex of $G_k$ has a loop, $G_k$ is 3-balanced and essentially $t$-balanced for all $t<|V|$.
However, since $G_k$ is not redundantly $\mathcal{L}_2$-rigid, it is not globally $\mathcal{L}_2$-rigid by Theorem \ref{thm:glob_char}.
\begin{figure}[ht]
\begin{center}

\begin{tikzpicture}[scale=1,font=\small]
%\fill[pattern=north west lines,opacity=.6,draw] (0,0) circle [radius=4];

\foreach \i in {1,2,...,10}
{
	\node[roundnode] (\i) at ({(\i-1)*36}: 1cm)  [] {};
}
\foreach \i/\j in {1/2,2/3,3/4,4/5,5/6,6/7,7/8,8/9,9/10,10/1}
{
	\draw[-] (\i) to (\j);
	\pgfmathsetmacro\startp{(\i-1)*36-20}
	\pgfmathsetmacro\endp{(\i-1)*36+20}
	\draw[-, in=\startp, out=\endp,loop] (\i) to (\i);
}
%\node[roundnode] at (-0.7cm, 0cm) (u) [] {};
%\node[roundnode] at (0.7cm, 0cm) (v) [] {};
%\node[roundnode] at (0cm, 0.7cm) (a) [] {}
%	edge[] (u)
%	edge[] (v);
\end{tikzpicture}
\end{center}
\caption{The graph $G_{10}$.}
\label{fig:G10}
\end{figure}
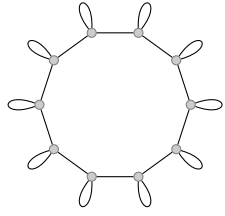

Since the graph $G_k$ is $\mathcal{L}_2$-rigid, one can wonder whether the conditions it satisfies are sufficient for being $\mathcal{L}_2$-rigid. However,
the following example shows that we can even construct flexible examples.\\

\noindent\textbf{Example 2.} Let $K_t$ be the complete graph whose vertex set is $\{v_1,v_2,\ldots,v_t\}$ and $t\geq 3$.
For each $1\leq i\leq t$, add a triangle whose vertex set is $\{x_i,y_i,z_i\}$. For each $1\leq i\leq t$ add the simple edge $v_ix_i$ and one loop at
each of $y_i,z_i$, see Figure \ref{fig:H9}. Let $H_t=(V,E,L)$ be the resulting looped simple graph. It is easy to deduce that $H_t$ is 3-balanced and essentially $t$-balanced.
To see that $H_t$ is not even $\mathcal{L}_2$-rigid, let $\X=\{X_0\}\cup\{\{v_1,v_2,\ldots,v_t\}\}\cup\{\{x_i,y_i,z_i\}:1\leq i\leq t\}\cup\{\{v_i,x_i\}:1\leq i\leq t\}$ where $X_0=\emptyset$.
Then $\X$ is an admissible 1-thin cover of $H-L$ whose looped member is the empty set. We can use \cite{G} and the facts that $|V|=4t$ and $|L|=2t$ to deduce that
$$
r_2^{lc}(H_t)\leq |L|+(2t-3)+(3\cdot t)+t=8t-3=2|V|-3,
$$
implying that $H_t$ is not $\mathcal{L}_2$-rigid.
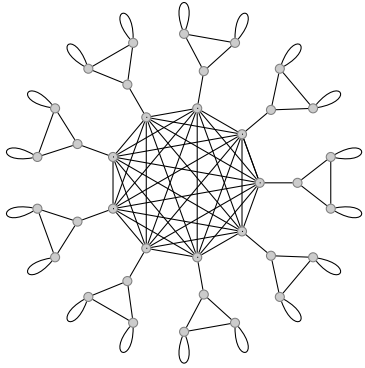
\begin{figure}[ht]
\begin{center}
\begin{tikzpicture}[scale=1,font=\small]
%\fill[pattern=north west lines,opacity=.6,draw] (0,0) circle [radius=4];

\foreach \i in {1,2,...,9}
{
	\node[roundnode] (\i) at ({(\i-1)*40}: 1cm)  [] {};
}
\foreach \i in {1,2,...,9}{	
	\foreach \j in {1,2,...,\i}{
	\draw[-] (\i) to (\j);}
	\pgfmathsetmacro\startpy{(\i-1)*40-30}
	\pgfmathsetmacro\endpy{(\i-1)*40+10}
	\pgfmathsetmacro\startpz{(\i-1)*40-10}
	\pgfmathsetmacro\endpz{(\i-1)*40+30}
	\begin{scope}[shift=({(\i-1)*40}:1cm)]
		\node[roundnode] (x\i) at ({(\i-1)*40}: .5cm)  [] {}
			edge[] (\i);
		\node[roundnode] (y\i) at ({(\i-1)*40-20}: 1cm)  [] {}
			edge[] (x\i)
			edge[in=\startpy, out=\endpy, loop] ();
		\node[roundnode] (z\i) at ({(\i-1)*40+20}: 1cm)  [] {}
			edge[] (x\i)
			edge[] (y\i)
			edge[in=\startpz, out=\endpz, loop] ();
	\end{scope}
}
\end{tikzpicture}
\end{center}
\caption{The graph $H_9$.}
\label{fig:H9}
\end{figure}

The next example shows that we cannot decrease the number $2t$ in Theorem \ref{thm:d_geq2_main}, i.e., (weak) $2t$-balancedness even if we add some essential $m$-balancedness condition for some $m<|V|$.

\noindent\textbf{Example 3.} For $s\geq 1$ and $k\geq 3$, let $T_k^s$ be a looped simple graph which we obtain as follows
\begin{itemize}
    \item[i] first construct a rooted tree with $s+1$ levels such that the root vertex (which is in level 0) has $k$ descendants in level 1,
    \item[ii] the number of descendants in level $l+1$ of a vertex $x$ in level $l$ with $1\leq l\leq s-1$, is $k-2$ when $l$ is odd and $k-1$ when $l$ is even 
    \item[iii] add a single loop at each vertex in odd levels,
\end{itemize}
see Figure \ref{fig:rootedtrees} for some examples. We will next construct a looped simple graph $G_k^s$ from $T_k^s$. When $s$ is odd, take $k-1$ copies of $T_k^s$; and when $s$ is even, take $k$ copies of $T_k^s$. Identify the vertices in level $s+1$ with their copies and delete multiples of loops so that each vertex has at most one loop in $G_k^s$ (when $s$ is even the vertices at level $s+1$ will have loops). See Figure \ref{fig:rootedtreescombined} for some examples of $G_k^s$. It is easy to see that the loopless vertices in $G_k^s$ have degree $k$ whereas the looped vertices have degree $k+1$. Thus the average degree of $G_k^s$ is strictly less than $k+1$ and this gives
\begin{equation}\label{eq:EplusL_bound}
    |E(G_k^s)\cup L(G_k^s)|<\frac{|V(G_k^s)|(k+1)}{2}.
\end{equation}

Note that $G_k^s$ contains at least $k$ loops and is $(k-1)$-connected since for all distinct vertices $u$ and $v$ we can find a vertex disjoint $uv$-path by using different copies of $T_k^s$ and there are $k-1$ or $k$ such $T_k^s$. Thus $G_k^s$ is $(k-1)$-balanced. We can also see that vertex cuts of size $k-1$ are the ones that separate the looped vertices. Hence $G_k^s$ is in fact $k$-balanced. Moreover using the fact that for each edge $uv\in G_k^s$ either $u$ or $v$ has a loop, we immediately deduce that $G_k^s$ is essentially $m$-balanced for all $m<|V(G_k^s)|$.

Now consider $H=G_{2t-1}^s=(V,E,L)$ for some $s\geq 1$. Then using the arguments above we see that $H$ is $(2t-1)$-balanced and essentially $m$-balanced for all $m<|V|$. We can also obtain $|E|+|L|<t|V|$ by Eq (\ref{eq:EplusL_bound}). This implies that $H^{[d-t]}$ satisfies 
$$|E(H^{[d-t]})|+|L(H^{[d-t]})|<d|V|$$
implying $H^{[d-t]}$ does not have enough edges and loops to be $\mathcal{L}_d$-rigid.
\begin{figure}[ht]
\begin{center}
\begin{tikzpicture}[scale=1,font=\small]
%\fill[pattern=north west lines,opacity=.6,draw] (0,0) circle [radius=4];

\node[roundnode] (0) at (0, 1.5)  [] {};
\foreach \i in {1,2,3}
{
    \node[roundnode] (\i) at ({\i-2},0)  [] {}
    edge[] (0)
    edge[in=50, out=10, loop] ();
}

\foreach \i in {1,2,3}
{
        \node[roundnode] (\i1) at ({\i-2},-1)  [] {}
        edge[] (\i);
}

\foreach \i in {1,2,3}
{   
    \foreach \j in {1,2}
    {
        \node[roundnode] (\i1\j) at ({\i-1.25-\j*0.5},-2)  [] {}
        edge[] (\i1)
        edge[in=50, out=10, loop] ();

        \node[roundnode] (\i1\j1) at ({\i-1.25-\j*0.5},-3)  [] {}
        edge[] (\i1\j);
    }
}
\begin{scope}[xshift=6cm]
\node[roundnode] (0) at (0, 1.5)  [] {};
\foreach \i in {1,2,3,4}
{
    \node[roundnode] (\i) at ({\i*2-5},0)  [] {}
    edge[] (0)
    edge[in=70, out=110, loop] ();
    \foreach \j in {1,2}
    {
        \node[roundnode] (\i\j) at ({\i*2-3.5-\j},-1)  [] {}
        edge[] (\i);

        \foreach \k in {1,2,3}
        {
            \node[roundnode] (\i\j\k) at ({\i*2-3.5-\j-\k*.4+.8},-2)  [] {}
            edge[] (\i\j)
            edge[in=250, out=290, loop] ();
        
        }  
    }
}    
\end{scope}
\end{tikzpicture}
\end{center}
\caption{The graphs $T_3^4$ (left) and $T_4^3$ (right).}
\label{fig:rootedtrees}
\end{figure}
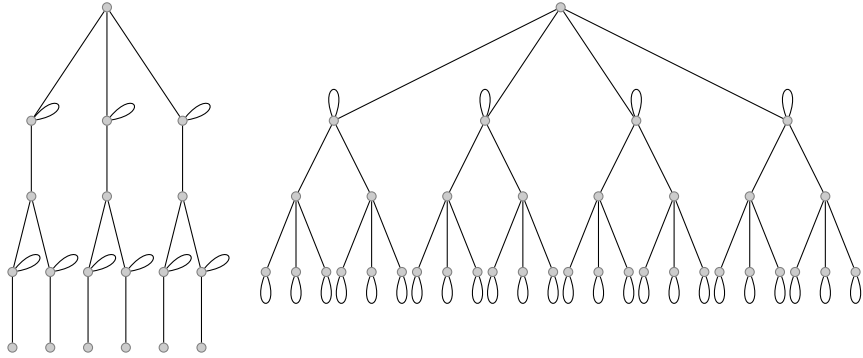

\begin{figure}[ht]
\begin{center}
\begin{tikzpicture}[scale=.8,font=\small]
%\fill[pattern=north west lines,opacity=.6,draw] (0,0) circle [radius=4];

\begin{scope}[rotate=90]
\node[roundnode] (0) at (0, 1.5)  [] {};
\foreach \i in {1,2,3}
{
    \node[roundnode] (\i) at ({\i-2},0)  [] {}
    edge[] (0)
    edge[in=50, out=10, loop] ();
}

\foreach \i in {1,2,3}
{
        \node[roundnode] (\i1) at ({\i-2},-1)  [] {}
        edge[] (\i);
}

\foreach \i in {1,2,3}
{   
    \foreach \j in {1,2}
    {
        \node[roundnode] (\i1\j) at ({\i-1.25-\j*0.5},-2)  [] {}
        edge[] (\i1)
        edge[in=50, out=10, loop] ();

        \node[roundnode] (\i1\j1) at ({\i-1.25-\j*0.5},-3)  [] {}
        edge[] (\i1\j);
    }
}
\end{scope}

\begin{scope}[rotate=270,yshift=7cm,xshift=-1.5cm]
\node[roundnode] (0) at (0, 1.5)  [] {};
\foreach \i in {1,2,3}
{
    \node[roundnode] (\i) at ({\i-2},0)  [] {}
    edge[] (0)
    edge[in=50, out=10, loop] ();
}

\foreach \i in {1,2,3}
{
        \node[roundnode] (\i1) at ({\i-2},-1)  [] {}
        edge[] (\i);
}

\foreach \i in {1,2,3}
{   
    \foreach \j in {1,2}
    {
        \pgfmathsetmacro\ireverse{int(4-\i)}
        \pgfmathsetmacro\jreverse{int(3-\j)}
        \node[roundnode] (\i1\j) at ({\i-1.25-\j*0.5},-2)  [] {}
        edge[] (\i1)
        edge[in=50, out=10, loop] ()
        edge[] (\ireverse1\jreverse1);
    }
}
\end{scope}
\begin{scope}[rotate=270,yshift=7cm,xshift=1.5cm]
\node[roundnode] (0) at (0, 1.5)  [] {};
\foreach \i in {1,2,3}
{
    \node[roundnode] (\i) at ({\i-2},0)  [] {}
    edge[] (0)
    edge[in=50, out=10, loop] ();
}

\foreach \i in {1,2,3}
{
        \node[roundnode] (\i1) at ({\i-2},-1)  [] {}
        edge[] (\i);
}

\foreach \i in {1,2,3}
{   
    \foreach \j in {1,2}
    {
        \pgfmathsetmacro\ireverse{int(4-\i)}
        \pgfmathsetmacro\jreverse{int(3-\j)}
        \node[roundnode] (\i1\j) at ({\i-1.25-\j*0.5},-2)  [] {}
        edge[] (\i1)
        edge[in=50, out=10, loop] ()
        edge[] (\ireverse1\jreverse1);
    }
}
\end{scope}
\begin{scope}[rotate=180,xshift=5cm,yshift=1.5cm]
\node[roundnode] (0) at (0, 1)  [] {};
\foreach \i in {1,2,3}
{
    \node[roundnode] (\i) at ({\i-2},0)  [] {}
    edge[] (0)
    edge[in=50, out=10, loop] ();
}

\foreach \i in {1,2,3}
{
        \pgfmathsetmacro\ireverse{int(4-\i)}
        \node[roundnode] (u\ireverse) at ({\i-2},-.5)  [] {}
        edge[] (\i);
}

\end{scope}
\begin{scope}[xshift=-5cm,yshift=1.5cm]
\node[roundnode] (0) at (0, 1)  [] {};
\foreach \i in {1,2,3}
{
    \node[roundnode] (\i) at ({\i-2},0)  [] {}
    edge[] (0)
    edge[in=50, out=10, loop] ();
}

\foreach \i in {1,2,3}
{
        \node[roundnode] (\i1) at ({\i-2},-.5)  [] {}
        edge[] (\i);
}

\foreach \i in {1,2,3}
{   
    \foreach \j in {1,2}
    {
        \node[roundnode] (\i1\j) at ({\i-1.25-\j*0.5},-1.5)  [] {}
        edge[] (\i1)
        edge[] (u\i)
        edge[in=50, out=10, loop] ();
    }
}
\end{scope}

\end{tikzpicture}
\end{center}
\caption{The graphs $G_3^3$ (left) and $G_3^4$ (right).}
\label{fig:rootedtreescombined}
\end{figure}
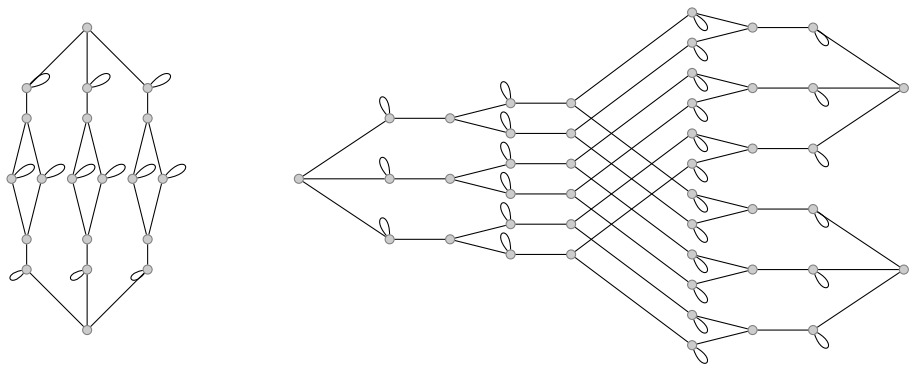

\section{Acknowledgements}
The first two authors are supported by T\"UB\.ITAK, grant no: 124F450.
The third author was partially supported by EPSRC grant EP/X036723/1 and by UK Research and Innovation (grant number UKRI1112), under the EPSRC Mathematical Sciences Small Grant scheme.
For the purpose of open access, the authors have applied a Creative Commons Attribution (CC-BY) licence to any Author Accepted Manuscript version arising.

\end{document}